\documentclass [11pt,a4paper]{article}
\usepackage{amsfonts,amssymb,amsmath,amscd,latexsym,makeidx,theorem,color,hyperref}
\usepackage[utf8]{inputenc}

\author{Gabriele Mancini\thanks{This work was supported by the Swiss National Science Foundation projects n. PP00P2-144669, PP00P2-170588/1 and P2BSP2-172064.} \\ \footnotesize Università Sapienza di Roma \\ \footnotesize \texttt{gabriele.mancini@uniroma1.it} \and  Luca Martinazzi${}^*$ \\ \footnotesize Università degli Studi di Padova \\ \footnotesize \texttt{luca.martinazzi@math.unipd.it}}
\title{Extremals for fractional Moser-Trudinger inequalities in dimension 1 via harmonic extensions and commutator estimates}
\date{}
\newtheorem{trm}{Theorem}[section]
\newtheorem{prop}[trm]{Proposition}
\newtheorem{cor}[trm]{Corollary}
\newtheorem{lemma}[trm]{Lemma}

\newtheorem{rmk}[trm]{Remark}
\newtheorem{thmx}{Theorem}[section]

\hypersetup{
    colorlinks,
    citecolor=black,
    linkcolor=black,
    urlcolor=black
}

\newcommand{\R}{\mathbb{R}}
\newcommand{\N}{\mathbb{N}}
\newcommand{\de}{\partial}
\newcommand{\ve}{\varepsilon}
\newcommand{\eps}{\varepsilon}
\newcommand{\rw}{\rightharpoonup}

\newcommand{\bra}[1]{\left({#1}\right)}

\newcommand{\hl}{(-\Delta)^\frac{1}{2}}
\newcommand{\ql}{(-\Delta)^\frac{1}{4}}

\newenvironment{proof}{\noindent\emph{Proof.}}{\phantom{ } \hfill$\square$\medskip}

\DeclareMathOperator{\loc}{loc}

\newcommand{\blu}[1]{{\color{blue}#1}}

\newcommand{\wt}{\widetilde}

\newcommand{\bdm}{\begin{displaymath}}
\newcommand{\edm}{\end{displaymath}}

\newcommand{\ph}{\varphi}

\newcommand{\ov}{\overline}

\begin{document}
\maketitle

\begin{abstract}
We prove the existence of extremals for fractional Moser-Trudinger inequalities in an interval and on the whole real line. In both cases we use blow-up analysis for the corresponding Euler-Lagrange equation, which requires new sharp estimates obtained via commutator techniques.
\end{abstract}

\section{Introduction}



The celebrated Moser-Trudinger inequality \cite{mos} states that for $\Omega\subset\R^n$ with finite measure $|\Omega|$ we have
\begin{equation}\label{stimaMT0}
\sup_{u\in W^{1,n}_0(\Omega),\;\|\nabla u\|_{L^n(\Omega)}\le 1} \int_\Omega e^{\alpha_n |u|^\frac{n}{n-1}}dx \le C|\Omega|,\quad \alpha_n:=n\omega_{n-1}^\frac{1}{n-1},
\end{equation}
where $\omega_{n-1}$ is the volume of the unit sphere in $\R^n$. The constant $\alpha_n$ is sharp in the sense that the supremum in \eqref{stimaMT0} becomes infinite if $\alpha_n$ is replaced by any $\alpha >\alpha_n$.  
In the case $\Omega=\R^2$, B. Ruf \cite{ruf} proved a similar inequality, using the full $W^{1,2}$-norm instead of the $L^2$-norm of the gradient, then generalized to $\R^n$, $n\ge 2$ by Li-Ruf \cite{LiRuf} as
\begin{equation}\label{stimaMTRuf}
\sup_{u\in W^{1,n}(\R^n),\;\|u\|_{L^n(\R^n)}^n+\|\nabla u\|_{L^n(\R^n)}^n\le 1} \int_{\R^n} \left(e^{\alpha_n |u|^\frac{n}{n-1}}-1\right)dx <\infty.
\end{equation}
Higher-order versions of \eqref{stimaMT0} were proven by Adams \cite{Ada}  on the space $W^{k,\frac{n}{k}}(\Omega)$ for  $n>k\in \mathbb{N}$.

\medskip
In \cite{IMM} the authors proved the following $1$-dimensional fractional extension of the previous results (for the definition of $H^{\frac12,2}(\R)$ and $(-\Delta)^\frac14$ see \eqref{defHsp} in the Appendix).

\begin{thmx}\label{MT3}
Set $I:=(-1,1)\subset\R$ and $\tilde H^{\frac12,2}(I):=\{u\in H^{\frac12,2}(\R): u\equiv 0 \text{ on }\R\setminus I\}$. Then we have
\begin{equation}\label{stimaMT}
\sup_{u\in \tilde H^{\frac12,2}(I), \;\|(-\Delta)^{\frac{1}{4}}u\|_{L^{2}(I)}\leq 1}\int_{I}\left(e^{\alpha u^2}-1\right)dx =C_\alpha<\infty,\quad \text{for }\alpha\le\pi,
\end{equation}
and
\begin{equation}\label{stimaMTR}
\sup_{u\in H^{\frac12,2}(\mathbb{R}), \;\|u\|_{{H}^{\frac12,2}(\R{})}\leq 1} \int_{\mathbb{R}}\left( e^{\alpha u^{2}}-1 \right) dx =D_\alpha<\infty,\quad \text{for }\alpha\le\pi,
\end{equation}
where $\|u\|_{H^{\frac12,2}(\R{})}^2:=\|(-\Delta)^{\frac{1}{4}}u\|_{L^{2}{(\R)}}^2+\|u\|_{L^2(\R)}^2$. The constant $\pi$ is sharp in \eqref{stimaMT} and \eqref{stimaMTR}. 
\end{thmx}

More general results have recently appeared, see e.g. \cite{AT,FM, LamLu,mar,RS,Takah}, in which both the dimension and the (fractional) order of differentiability have been generalized. For instance, \eqref{stimaMT} and \eqref{stimaMTR} can be seen as $1$-dimensional cases of the more general results of \cite{LamLu,mar,FM} that hold in arbitrary dimension $n$.

\medskip

The existence of extremals for this kind of inequalities is a challenging question. Existence of extremals for \eqref{stimaMT0} was originally proven by L. Carleson and A. Chang \cite{CC} in the case of the unit ball, a fundamental result later extended by Struwe \cite{sIHP} and Flucher \cite{Flu} to the case of general bounded domains in $\R^2$ and by K. Lin \cite{Lin} to the case of bounded domains in $\R^n$. In the case of the Li-Ruf inequality \eqref{stimaMTRuf}, the existence of extremals appears in \cite{LiRuf} when $n\ge 3$ and was proven by Ishiwata \cite{Ish} when $n=2$. For the higher-order Adams inequality the existence of extremals  has been proven in various cases, e.g. by Li-Ndiaye \cite{LN} on a $4$ dimensional closed manyfold, by Lu-Yang \cite{LuYang} for a $4$ dimensional bounded domain and by DelaTorre-Mancini \cite{DlTM} for a bounded domain in $\R^{2m}$, $m\ge 1$ arbitrary. 

On the other hand, the existence of extremals for the fractional Moser-Trudinger inequality has remained open until now,  with the exception of Takahashi \cite{Takah} considering a subcritical version of \eqref{stimaMTR} of Adachi-Tanaka type \cite{AT}, and Li-Liu \cite{LL} treating the case of a fractional Moser-Trudinger on $H^{\frac{1}{2},2}(\de M)$ with $M$ a compact Riemann surface with boundary. 
The idea of Li and Liu is that working on the boundary of a compact manifold, one can localize the $H^{\frac{1}{2},2}$-norm. 

Applying the same method for an interval $I\subset \R$ creates problems near $\de I$, which require additional care in the estimate, and the problem becomes even more challenging when working on the whole $\R$. The main purpose of this paper is to handle these two cases and prove that the suprema in \eqref{stimaMT} and \eqref{stimaMTR} are attained.

\begin{trm}\label{extI}
For any $0<\alpha \le \pi$, the inequality \eqref{stimaMT} has an extremal i.e. there exists $u_\alpha \in \tilde H^{\frac{1}{2},2}(I)$ such that  $\|(-\Delta)^\frac{1}{4} u_\alpha\|_{L^2(\R)}\le1$ and
$$
\int_{I} \bra{e^{\alpha u_\alpha^2} -1} dx = C_\alpha. 
$$
\end{trm}

Theorem \ref{extI} is rather simple to prove for $\alpha \in (0,\pi)$,  while the case $\alpha = \pi$ relies on a delicate blow-up analysis for subcritical extremals. 

 A similar analysis can be carried out for the Ruf-type inequality \eqref{stimaMTR}.  However, working on the whole real line we need to face  additional difficulties due to the lack of compactness of the embedding of $H=H^{\frac{1}{2},2}(\R)$ into $L^2(\R)$: vanishing at infinity might occur for maximizing sequences, even in the sub-critical case $\alpha \in (0,\pi)$.  This issue is not merely technical indeed Takahashi \cite{Takah} proved that \eqref{stimaMTR} has no extremal when $\alpha$ is small enough. Here, in analogy with the results in dimension $n\ge 2$, we prove that the supremum in \eqref{stimaMTR} is attained if $\alpha$ sufficiently close to $\pi$.


\begin{trm}\label{extR}
There exists $\alpha^*\in (0,\pi)$ such that for $\alpha^*\le \alpha\le \pi$ the inequality \eqref{stimaMTR} has an extremal, namely,  there exists $\bar u_\alpha \in  H^{\frac{1}{2},2}(\R)$ such that $\|\bar u_\alpha\|_{H^{\frac{1}{2},2}(\R)}\le 1$ and 
$$
\int_{\R} \bra{e^{\alpha \bar u^2_\alpha} -1} dx = D_\alpha. 
$$
\end{trm}

As for Theorem \ref{extI}, the proof of Theorem \ref{extR} for $\alpha =\pi$ is based on blow-up analysis. In fact we need to study the blow-up of a non-local equation on the whole real line (no boundary conditions), as done in the following theorem.


\begin{trm}\label{compact}
Let $(u_k)\subset H=H^{\frac12,2}(\R)$ be a sequence of non-negative solutions to
\begin{equation}\label{eqR}
(-\Delta)^\frac12 u_k+u_k=\lambda_k u_k e^{\alpha_k u_k^2}\quad \text{in }\R,
\end{equation}
where $\alpha_k\to \pi$ and $\lambda_k \to \lambda_\infty\ge 0$.
{Assume $u_k$ even and decreasing ($u_k(-x)=u_k(x)\le u_k(y)$ for $x\ge y\ge 0$) for every $k$} and set $\mu_k:=\sup_\R u_k=u_k(0)$. Assume also that
\begin{equation}\label{energybound}
\Lambda:=\limsup_{k\to\infty}\|u_k\|_H^2<\infty.
\end{equation}
Then up to extracting a subsequence we have that either
\begin{itemize}
\item[(i)] $\mu_k\le C$, $u_k\to u_\infty$ in $C^\ell_{\loc}(\R)$ for every $\ell\ge 0$, where $u_\infty\in C^\ell_{\loc}(\R)\cap H$ solves
\begin{equation}\label{eqinftyR}
(-\Delta)^\frac12 u_\infty+u_\infty=\lambda_\infty u_\infty e^{\pi u_\infty^2}\quad \text{in }\R,
\end{equation}
or
\item[(ii)] $\mu_k\to \infty$, $u_k\to u_\infty$ weakly in $H$ and strongly in $C^0_{\loc}(\bar \R \setminus \{0\})$ where $u_\infty$ is a solution to \eqref{eqinftyR}. Moreover, setting $r_k$ such that
\begin{equation}\label{defmukrk}
\lambda_k r_k \mu_k^2 e^{\alpha_k \mu_k^2} =\frac{1}{\alpha_k},
\end{equation}
and
\begin{equation}\label{etak}
\eta_k(x):=2\alpha_k\mu_k(u_k(r_kx)-\mu_k),\quad \eta_\infty(x):=-\log\left({1+|x|^2}\right),
\end{equation}
one has $\eta_k\to \eta_\infty$ in $C^\ell_{\loc}(\R{})$ for every $\ell\ge 0$, $\sup_{k}\|\eta_k\|_{L_s(\R)}<\infty$ for any $s>0$ (cfr. \eqref{L12}),  and $\Lambda\ge \|u_\infty\|_H^2+1$. 
\end{itemize}
\end{trm}

The proof of Theorem \ref{compact} is quite delicate because local elliptic estimates of a nonlocal equation depend on global bounds as we shall prove in Lemma \ref{LemmaLs}. This will be based on sharp commutator estimates (Lemma \ref{Strange Lemma}), as developed in \cite{MMS} for the case of a bounded domain in $\R^n$, extending to the fractional case the approach of \cite{mar2}.
 
 \medskip

We expect similar existence results to hold for a perturbed version of inequalities \eqref{stimaMT}-\eqref{stimaMTR}, as in \cite{MM} and \cite{Thi} (see also the recent results in \cite{IMNS}), but we will not investigate this issue here.  

\section{Proof of Theorem \ref{extI}}\label{SecExtI}

\subsection{Strategy of the proof}
We will focus on the case $\alpha =\pi$, since the existence of extremals for \eqref{stimaMT} with $\alpha\in (0,\pi)$  follows easily by Vitali's convergence theorem, see e.g. the argument in \cite[Proposition 6]{MM}.

Let $u_k$ be an extremal of \eqref{stimaMT} for $\alpha =\alpha_k=\pi-\frac1k$.  
By replacing $u_k$ with $|u_k|$ we can assume that $u_k\ge 0$. Moreover $\|(-\Delta)^\frac{1}{4} u_k\|_{L^2(\R)}=1$, and $u_k$ satisfies the Euler-Lagrange equation
\begin{equation}\label{eqI}
\hl u_k=\lambda_k u_k e^{\alpha_k u_k^2},
\end{equation}
with bounds on the Lagrange multipliers $\lambda_k$ (see \eqref{lambdak}).

Using the monotone convergence theorem we also get
\begin{equation}\label{convext}
\lim_{k\to\infty}\int_{I}\left(e^{\alpha_k u_k^2}-1\right)dx =\lim_{k\to \infty} C_{\alpha_k} = C_\pi,
\end{equation}
where $C_{\alpha_k}$ and $C_\pi$ are as in \eqref{stimaMT}.

If $\mu_k:=\max_I u_k=O(1)$ as $k\to\infty$, then up to a subsequence $u_k\to u_\infty$ locally uniformly, where by \eqref{convext} $u_\infty$ maximizes \eqref{stimaMT} with $\alpha=\pi$. Therefore we will work by contradiction, assuming
\begin{equation}\label{ukinfty}
\lim_{k\to\infty} \mu_k= \infty.
\end{equation}
By studying the blow-up behavior of $u_k$, see in particular Propositions \ref{conveta} and \ref{convgreen}, we will show that \eqref{ukinfty} implies  $C_\pi\le 4\pi$ (Proposition \ref{Cg1}), but with suitable test functions we will also prove that $C_\pi>4\pi$ (Proposition \ref{Cg2}), hence contradicting \eqref{ukinfty} and completing the proof of Theorem \ref{extI}. 

\subsection{The blow-up analysis}

The following proposition is well known in the local case, and its proof in the present setting is similar to the local one. We give it for completeness.

\begin{prop}
We have $u_k\in C^\infty(I)\cap C^{0,\frac12}(\bar I)$, $u_k>0$ in $I$, and $u_k$ is symmetric with respect to $0$ and decreasing with respect to $|x|$. Moreover,
\begin{equation}\label{lambdak}
0<\lambda_k <\lambda_1(I).
\end{equation}
Up to a subsequence we have $\lambda_k\to \lambda_\infty$ and $u_k\to u_\infty$ weakly in $\tilde H^{\frac12,2}(I)$ and strongly in $L^2(I)$, where $u_\infty$ solves
\begin{equation}\label{eqinfty}
\hl u_\infty=\lambda_\infty u_\infty e^{\pi u_\infty^2}.
\end{equation}
\end{prop}

\begin{proof}
For the first claim see Remark 1.4 in \cite{MMS}. The positivity follows from the maximum principle, and symmetry and monotonicity follow from the moving point technique, see e.g. \cite[Theorem 11]{DHMS}.  

Now testing \eqref{eqI} with $\varphi_1$, the first eigenfunction of $\hl$ in $\tilde H^{\frac12,2}(I)$, positive and with eigenvalue $\lambda_1(I)>0$, we obtain
$$\lambda_1(I)\int_I u_k\varphi_1dx=\lambda_k \int_I u_k e^{\alpha_k u_k^2}\varphi_1dx > \lambda_k \int_I u_k\varphi_1 dx,$$
hence proving \eqref{lambdak}. By the theorem of Banach-Alaoglu and the compactness of the Sobolev embedding of $\tilde H^{\frac12,2}(I)\hookrightarrow L^2(I)$, we obtain the claimed convergence of $u_k$ to $u_\infty$. Finally, to show that $u_\infty$ solves \eqref{eqinfty}, test with $\varphi\in C^\infty_c(I)$:
\[\begin{split}
\int_I u_\infty \hl \varphi dx&=\lim_{k\to\infty}\int_I u_k \hl \varphi dx\\
&=\lim_{k\to\infty}\int_I \lambda_k u_ke^{\alpha_k u_k^2}\varphi dx\\
&=\int_I \lambda_\infty u_\infty e^{\pi u_\infty^2}\varphi dx,
\end{split}\]
where the convergence of the last integral is justified by splitting $I$ into $I_1:=\{x\in I: u_k(x)\le L\}$ and $I_2:=\{x\in I:u_k(x)>L\}$, applying the dominated convergence on $I_1$ and bounding
\[\begin{split}
\int_{I_2} \lambda_ku_ke^{\alpha_k u_k^2}\varphi \, dx& \le \frac{\sup_I{|\varphi|}}{L}\int_{I}\lambda_k u_k^2e^{\alpha_k u_k^2}\,dx\\
&=\frac{\sup_I{|\varphi|}}{L}\int_{I} u_k \hl u_k\, dx\\
&=\frac{\sup_I{|\varphi|}}{L} \|\ql u_k\|_{L^2(\R)}^2,
\end{split}\]
and letting $L\to\infty$.
\end{proof}

Let $\tilde u_k$ be the harmonic extension of $u_k$ to $\R^2_+$ given by the Poisson integral, see \eqref{Poisson2} in the appendix.  Notice that
\begin{equation}\label{energy}
\int_I \lambda_k u_k^2e^{\alpha_k u_k^2}dx= \|\ql u_k\|_{L^2(\R)}^2=\|\nabla \tilde u_k\|_{L^2(\R^2_+)}^2=1.
\end{equation}
Let $r_k = \frac{1}{\alpha_k\lambda_k\mu_k^2e^{\alpha_k\mu_k^2}}$ and $\eta_k(x):= 2\alpha_k \mu_k(u_k(r_k x)-\mu_k)$ be as in \eqref{defmukrk} and \eqref{etak}, and set
$$\tilde \eta_k(x,y):= {2\alpha_k } \mu_k(\tilde u_k(r_k x,r_k y)-\mu_k).$$
Note that $\tilde \eta_k$ is the Poisson integral of $\eta_k$. 

\begin{prop}\label{conveta}
We have $r_k\to 0$ and $\tilde \eta_k\to \tilde \eta_\infty$ in $C^\ell_{\loc}(\overline{\R^2_+})$ for every $\ell\ge 0$, where
$$\tilde \eta_\infty(x,y)=-\log\bra{(1+y)^2+x^2}$$
is the Poisson integral (compare to \eqref{Poisson2}) of $\eta_\infty:=-\log\bra{1+x^2}$, and
\begin{equation}\label{propeta}
\hl \eta_\infty=2e^{\eta_\infty},\quad \int_{\R}e^{\eta_\infty}dx=\pi.
\end{equation}
\end{prop}

\begin{proof} According to Lemma 2.2, Theorem 1.5 and Proposition 2.7 in \cite{MMS}, we have $r_k\to 0$,  $\eta_k\to \eta_\infty$ in $C^\ell_{\loc}(\R)$ for every $\ell\ge 0$ and $(\eta_k)$ is uniformly bounded in $L_\frac12(\R)$ (see \eqref{L12}).  

To obtain the local convergence of $\tilde \eta_k$, fix $R>0$ and split the integral in the Poisson integral \eqref{Poisson2} of $\tilde \eta_k$ into an integral over $(-R,R)$ and an integral over $\R\setminus (-R,R)$, for $R$ large. The former is bounded by the convergence of $\eta_k$ locally, the latter by the boundedness of $\eta_k$ in $L_\frac12(\R)$, provided $(x,y)\in B_{\frac{R}{2}} \cap \R_+^2$. As a consequence we get that $\tilde \eta_k$ is locally uniformly bounded in $\overline{\R^2_+}$. Since $\tilde \eta_k$ is harmonic, we conclude by elliptic estimates.
\end{proof}

\begin{cor}\label{uinfty0} 
For $R>0$ and $i=0,1,2$, we have  
\begin{equation}\label{integrals}
\lim_{k\to \infty}\int_{-R r_k}^{R r_k} \lambda_k \mu_k^i u_k^{2-i} e^{\alpha_k u_k^2} dx =  \frac{1}{\pi} \int_{-R}^R e^{\eta_\infty} dx. 
\end{equation}
Moreover, $u_\infty\equiv 0$, i.e. up to a subsequence $u_k\to 0$ in $L^2(I)$, weakly in $\tilde H^{\frac12,2}(I)$, and a.e in $I$.
\end{cor}
\begin{proof}
With the change of variables $\xi=\frac{x}{r_k}$, writing $u_k(r_k\cdot )=\mu_k+\frac{\eta_k}{2\alpha_k\mu_k}$ and using \eqref{defmukrk} and Proposition \ref{conveta}, we see that
\[
\int_{-R r_k}^{R r_k} \lambda_k \mu_k^i u_k^{2-i} e^{\alpha_k u_k^2} dx = \underbrace{ r_k \lambda_k \mu_k^2 e^{\alpha_k \mu_k^2}  }_{=\frac{1}{\alpha_k}} \int_{-R}^R  \left(1+\frac{\eta_k}{2\alpha_k\mu_k^2}\right)^{2-i} e^{\eta_k + \frac{\eta_k}{4\alpha_k \mu_k^2}} d\xi\to \frac{1}{\pi}\int_{-R}^R e^{\eta_\infty} d\xi,
\]
as $k\to \infty$, as claimed in \eqref{integrals}.

In order to prove the last statement, recalling that $\|\ql u_k\|_{L^2}=1$, we write
$$1=\int_{-Rr_k}^{Rr_k}\lambda_k u_k^2 e^{\alpha_ku_k^2}dx+\int_{I\setminus (-Rr_k,Rr_k)}\lambda_k u_k^2 e^{\alpha_ku_k^2} dx=:(I)_k+(II)_k.$$
By \eqref{integrals} and \eqref{propeta} we get 
\[\begin{split}\lim_{k\to\infty}(I)_k   = \frac{1}{\pi}\int_{-R}^R e^{\eta_\infty}dx =1+o(1),
\end{split}\]
with $o(1)\to 0$ as $R\to\infty$. This in turn implies that
$$\lim_{R\to\infty}\lim_{k\to\infty}(II)_k= 0,$$
which  is possible only if $u_\infty\equiv 0$, or $\lambda_\infty=0$ (by Fatou's lemma). But on account of \eqref{eqinfty}, also in the latter case we have $u_\infty\equiv 0$.
\end{proof}


\begin{lemma}\label{lemma1A}
For $A>1$, set  $u_k^A:=\min\left\{u_k,\frac{\mu_k}{A}\right\}$. Then we have
\begin{equation}\label{A1}
\limsup_{k\to \infty}\|\ql u_k^A\|_{L^2(\R)}^2 \le \frac{1}{A}.
\end{equation}
\end{lemma}

\begin{proof}
We set $\bar u_k^A:=\min\left\{\tilde u_k,\frac{\mu_k}{A}\right\}$. 
Since $\bar u_{k}^A$ is an extension {(in general not harmonic)} of $u_k^A$, we have 
\begin{equation}\label{toprove1}
\|\ql u_k^A\|_{L^2(\R)}^2  \le \int_{\R^2_+} |\nabla \bar u_k^A|^2 dxdy.
\end{equation}
Using integration by parts and the harmonicity of $\tilde u_k$ we get
\begin{equation}\label{toprove2}
\begin{split}
\int_{\R^2_+}|\nabla \bar u_k^A|^2dx dy&=\int_{\R^2_+}\nabla  \bar u_k^A \cdot \nabla \tilde u_k dxdy\\
&= - \int_{\R} u_k^A(x) \frac{\de \tilde u_k(x,0)}{\de y}dx\\
&=\int_{\R{}}(-\Delta)^{\frac 12} u_k u_k^Adx.
\end{split}
\end{equation}
Proposition \ref{conveta} implies that $u_k^A(r_kx)=\frac{\mu_k}{A}$ for $|x|\le R$ and $k\ge k_0(R)$. Then, with  \eqref{propeta} and \eqref{integrals}
we obtain 
\[
\begin{split}
\int_{\R{}}(-\Delta)^{\frac 12} u_k u_k^A\, dx&\geq \int_{-Rr_k}^{Rr_k} \lambda_k  u_k e^{\alpha_k u_k^2}u_k^A dx\\&
\stackrel{k\to\infty}{\to} \frac{1}{\pi A}\int_{-R}^R e^{\eta_\infty} d\xi\\
&\stackrel{R\to\infty}{\to} \frac{1}{A}.
\end{split}
\]

Set now $v_k^A:=\left(u_k-\frac{\mu_k}{A}\right)^+ = u_k - u_k^{A}$. With  similar computations we get
\[
\begin{split}
\int_{\R{}}(-\Delta)^{\frac 12} u_k v_k^A dx&\geq \int_{-Rr_k}^{Rr_k} \lambda_k u_k v_k^A e^{\alpha_k u_k^2} dx\\
&\stackrel{k\to\infty}{\to}\frac{1}{\pi}\bra{1-\frac{1}{A}} \int_{-R}^R e^{\eta_\infty}d\xi\\
&\stackrel{R\to\infty}{\to} \frac{A-1}{A}.
\end{split}
\]
Since
$$
\int_{\R{}}(-\Delta)^{\frac 12} u_k u_k^A\, dx+ \int_{\R{}}(-\Delta)^{\frac 12} u_k v_k^A dx=\int_{\R}(-\Delta)^{\frac 12} u_k u_k dx = 1,$$
we get that
\[
\lim_{k\to \infty} \int_{\R{}}(-\Delta)^{\frac 12} u_k u_k^A\, dx = \frac{1}{A}.
\] 
Then, we conclude using \eqref{toprove1}, and \eqref{toprove2}.
\end{proof}

\begin{prop}\label{stimaalto} We have
\begin{equation}\label{Cglim}
C_\pi= \lim_{k\to\infty}\frac{1}{\lambda_k\mu_k^2}.
\end{equation}
Moreover
\begin{equation}\label{muklambdak}
\lim_{k\to\infty} \mu_k \lambda_k =0.
\end{equation}
\end{prop}

\begin{proof} Fix $A>1$ and let $u_k^A$ be defined as in Lemma \ref{lemma1A}.  We split
\[\begin{split}
\int_{I}\bra{e^{\alpha_k u_k^2}-1}dx&= \int_{I\cap\{u_k\le\frac{\mu_k}{A}\}}\bra{e^{\alpha_k (u_k^A)^2}-1}dx
 +\int_{I\cap\{u_k>\frac{\mu_k}{A}\}}\bra{e^{\alpha_k u_k^2}-1}dx=:(I)+(II).
\end{split}\]
Using Corollary \ref{uinfty0} and Vitali's theorem, we see that
$$(I)\le  \int_{I}\bra{e^{\alpha_k (u_k^A)^2}-1}dx\to 0\quad \text{as }k\to\infty,$$
since $e^{\alpha_k (u_k^A)^2}$ is uniformly bounded in $L^A(I)$ by Lemma \ref{lemma1A} together with Theorem \ref{MT3}.

By \eqref{energy} and Corollary \ref{uinfty0}, we now estimate
\[
(II)\le \frac{A^2}{\lambda_k \mu_k^2} \int_{I\cap\{u_k>\frac{\mu_k}{A}\}}\lambda_k u_k^2 \left( e^{\alpha_k u_k^2}-1\right) dx
\le \frac{A^2}{\lambda_k \mu_k^2} (1+o(1)),
\]
with $o(1)\to 0$ as $k\to\infty$.
 Together with \eqref{convext}, and letting $A\downarrow 1$, this gives
$$C_\pi\le \lim_{k\to\infty}\frac{1}{\lambda_k\mu_k^2}.$$
The converse inequality follows from \eqref{integrals} as follows:
\[\begin{split}
\int_{I}\bra{e^{\alpha_k u_k^2}-1}dx&\ge \int_{-Rr_k}^{Rr_k}e^{\alpha_k u_k^2}dx +o(1)\\
&=\frac{1}{\lambda_k\mu_k^2} \left( \frac{1}{\pi} \int_{-R}^R e^{\eta_\infty}dx +o(1)\right) +o(1),
\end{split}\]
with $o(1)\to 0$ as $k\to\infty$. Letting $R\to\infty$ and recalling \eqref{propeta} we obtain \eqref{Cglim}.

Finally, \eqref{muklambdak} follows at once from \eqref{Cglim}, because otherwise we would have $C_\pi=0$, which is clearly impossible.
\end{proof}

\begin{prop}\label{convdelta0} 
Let us set 
$f_k :=\lambda_k \mu_k u_k e^{\alpha_k u_k^2}$. Then we have 
$$
\int_{I} f_k \ph \, dx \to \ph(0),
$$
as $k\to \infty$,  for any $\ph \in C(\bar I)$. In particular, $f_k \rightharpoonup \delta_0$ 
in the sense of Radon measures in $I$.
\end{prop}

\begin{proof} Take $\ph \in {C}(\bar I)$. For given $R>0$, $A>1$, we split
\[\begin{split}
\int_I \varphi f_k  dx &=\int_{-Rr_k}^{Rr_k} \varphi f_k \, dx+ \int_{\{u_k> \frac{\mu_k}{A}\}\setminus (-Rr_k, R r_k)} \varphi f_k  \, dx+ \int_{\{u_k\le  \frac{\mu_k}{A}\}} \varphi f_k \, dx\\
&=: I_1 +I_2+I_3.
\end{split}\]
On $\{u_k\le  \frac{\mu_k}{A}\}$ we have $u_k= u_k^A$ and Lemma \ref{lemma1A} and Theorem \ref{MT3} imply that $u_k e^{\alpha_k u_k^2}$ is uniformly bounded in $L^1$ (depending on $A$). Thus using \eqref{muklambdak} we get $I_3\to 0$.

With \eqref{energy} and \eqref{integrals} we also get
\[\begin{split}
I_2&\le A \|\varphi\|_{L^\infty(I)}\int_{\{u_k> \frac{\mu_k}{A}\}\setminus (-Rr_k,Rr_k )}\lambda_k u_k^2e^{\alpha_k u_k^2} \, dx\\
&\le A \|\varphi\|_{L^\infty(I)}\bra{1-\int_{-Rr_k}^{Rr_k}\lambda_k u_k^2e^{\alpha_k u_k^2}\, dx}\\
&=A \|\varphi\|_{L^\infty(I)}\bra{1-\frac{1}{\pi}\int_{-R}^Re^{\eta_\infty}dx+o(1)}
\end{split}\]
with $o(1)\to 0$ as $k\to \infty$. Thanks to \eqref{propeta}, we conclude that $I_2\to 0$ as $k\to\infty$ and $R\to \infty$.

As for $I_1$, again with \eqref{integrals} we compute
$$I_1=(\varphi(0)+o(1))\bra{\frac{1}{\pi}\int_{-R}^R e^{\eta_\infty}dx +o(1)},$$
so that $I_1\to\varphi(0)$ as $k\to\infty$ and $R\to\infty$.
\end{proof}

Given $x\in I$, let $G_x: \R\setminus \{0\} \to\R$ be the Green's function of $\hl$ on $I$ with singularity at $x$.  We recall that we have the explicit formula (see e.g. \cite{BGR})
\begin{equation}\label{FormulaGI}
G_x(y):= \begin{cases}
\frac{1}{\pi} \log \left(\frac{1-x y {+} \sqrt{(1-x^2)(1-y^2)}}{|x-y|}\right), &   y\in I,
\\ 0  & y \in \R \setminus I. 
\end{cases}
\end{equation}
In the following we further denote 
\begin{equation}\label{defS}
S(x,y):= G_{x}(y)- \frac{1}{\pi} \log \frac{1}{|x-y|}. 
\end{equation}


\begin{lemma}\label{leconvgreen} We have $\mu_k  u_k \rightarrow  G:=G_0$  in $L^\infty_{\loc}(\ov I \setminus \{0\}) \cap L^1(I)$ as $k\to +\infty$.
\end{lemma}

\begin{proof}
Let us set $v_k:= \mu_k u_k - G$ and $f_k = \mu_k \lambda_k u_k e^{\alpha_k u_k^2}$.  Arguing as in Proposition \ref{convdelta0}, we show that $\|f_k\|_{L^1(\R)}\to 1$ as $k\to \infty$. Moreover, since $u_k$ is decreasing with respect to $|x|$, we get that $u_k\to 0$ and $f_k\to 0$ locally uniformly in $\ov I \setminus \{0\}$ as $k\to \infty$.  By Green's representation formula, we have 
\begin{equation}\label{Green1}
\begin{split}
|v_k(x)| & =  \left| \int_{I} G_x(y) f_k(y)\,dy - G(x) \right| \\
 & \le \int_{I}  |G_x(y)-G(x)| f_k(y)  \,dy   + |\|f_k\|_{L^1(I)}-1|\, |G(x)|, \quad x\in I. 
\end{split}
\end{equation}
Fix $\sigma\in (0,1)$. If we assume  $|x|\ge \sigma$,   $|y|\le \frac{\sigma}{2}$, then we have 
\begin{equation}\label{Grenn2}\begin{split}
|G_x(y)-G(x)|& \le \frac{1}{\pi}\left|\log \frac{|x|}{|x-y|}\right| + |S(x,y)-S(x,0)| \\
& \le \frac{1}{\pi}\left|\log \left|\frac{x}{|x|} - \frac{y}{|x|} \right|\right|  + \sup_{|x|\ge \sigma, |y|\le  \frac{\sigma}{2}} |\nabla_y S(x,y)| |y| \\
& \le C |y|,
\end{split}\end{equation}
where $C$ is a constant depending only on $\sigma$. 
Then, for any $\eps \in (0,\frac{\sigma}{2})$,  we can write 
\begin{equation}\label{Green3}\begin{split}
|v_k(x)|& \le \int_{I}  |G_x(y)-G(x)| f_k(y)  \,dy + o(1) \\
&  = \int_{-\eps}^\eps  |G_x(y)-G(x)| f_k(y)  \,dy + \int_{I\setminus(-\eps,\eps)}  |G_x(y)-G(x)| f_k(y)  \,dy +o(1)\\
& \le C \eps  \|f_k\|_{L^1(-\eps,\eps)} +  \bra{\sup_{z\in I}\|G_z\|_{L^1(I)} +|G(x)|}\|f_k\|_{L^\infty(I \setminus (-\eps,\eps))}+o(1) \\
& \le  C \eps +o(1),
\end{split}\end{equation}
where $o(1)\to 0$ uniformly in $ I\setminus (-\sigma,\sigma)$ as $k\to \infty$. Clearly, \eqref{Green3} implies 
$$
\limsup_{k\to \infty} \|v_k\|_{L^\infty(I\setminus (-\sigma,\sigma))} \le C \eps. 
$$
Since $\eps$ and $\sigma$ can be  arbitrarily small, this shows that $v_k\to 0$ in $L^\infty_{\loc}(\ov I \setminus \{0\})$. With a similar argument, we prove the $L^1$ convergence. Indeed, integrating   \eqref{Green1}, for $\eps\in (0,1)$ we get
\begin{equation}\label{Green4}\begin{split}
\|v_k\|_{L^1(I)} &\le \int_I  \int_I |G_x(y)-G(x)|f_k(y ) \, dy\, dx  + | \|f_k\|_{L^1(I)} -1| \|G\|_{L^1(I)} \\
& \le  \int_I  f_k(y)  \int_{I}  |G_x(y)-G(x)| \,dx\, dy  + o(1)\\
& \le \int_{-\eps}^\eps  f_k(y)  \int_{I}  |G_x(y)-G(x)| \,dx\, dy +  2  \sup_{z\in I} \|G_z\|_{L^1(I)} \|f_k\|_{L^\infty(I\setminus(-\eps,\eps))} + o(1)\\
& = \int_{-\eps}^\eps  f_k(y)  \int_{I}  |G_x(y)-G(x)| \,dx\, dy +o(1). 
\end{split}\end{equation}
Since $$\sup_{y\in (-\eps,\eps)} \sup_{x\in I} |S(x,y)-S(x,0)| = O(\eps),$$
we get 
$$
\int_{-\eps}^\eps f_k(y) \int_{I} |G_y(x)-G(x)| \, dx\, dy = \frac{1}{\pi}\int_{-\eps}^\eps f_k(y) \int_{I} |\log\frac{|x|}{|x-y|}| \,dx\,dy + O(\eps).
$$
Moreover, using the change of variables $x = y z$, we obtain
$$
\int_{I} \left|\log\frac{|x|}{|x-y|}\right| \,dx = |y| \int_{-\frac{1}{|y|}}^\frac{1}{|y|}\ \left| \log \frac{|z|}{|z-1|}\right| dz = O\left(|y|\log\frac{1}{|y|}\right).
$$
Then, we have 
\begin{equation}\label{Green5}
\int_{-\eps}^\eps f_k(y) \int_{I} |G_y(x)-G_0(x)| \, dx\, dy = \int_{-\eps}^\eps f_k(y)  O\left(|y|\log\frac{1}{|y|}\right)\,dy + O(\eps)  = O\left(\eps \log \frac{1}{\eps}\right). 
\end{equation}
Clearly \eqref{Green4} and \eqref{Green5} yield $\limsup_{k\to +\infty} \|v_k -G\|_{L^1(I)} = O(\eps \log\frac{1}{\eps})$. Since $\eps$ can be arbitrarily small we get the conclusion. 
\end{proof}

\begin{prop}\label{convgreen} We have $\mu_k \tilde u_k \rightarrow \tilde G$  in $C^0_{\loc}(\overline{\R^2_+}\setminus\{(0,0)\}) \cap C^1_{\loc}( \R^2_+)$, where $\tilde G$ is the Poisson extension of $G$.
\end{prop}
\begin{proof}
As in the proof of Lemma \ref{leconvgreen}  we denote $v_k:= \mu_k u_k - G$. Let us consider the Poisson extension $\widetilde v_k= \mu_k \widetilde u_k - \widetilde G$. For any fixed $\eps>0$, we can split
$$
\wt v_k (x,y)= \frac{1}{\pi} \int_{-\eps}^\eps \frac{y v_k(\xi)}{(x-\xi)^2+y^2} d\xi + \frac{1}{\pi} \int_{I\setminus (-\eps,\eps)} \frac{yv_k(\xi)}{(x-\xi)^2+y^2} d\xi,
$$
By Lemma \ref{leconvgreen}, we have 
$$
 \left|\frac{1}{\pi}  \int_{I\setminus (-\eps,\eps)} \frac{y v_k(\xi)}{(x-\xi)^2+y^2} d\xi \right| \le \frac{1}{\pi}  \|v_k\|_{L^\infty(I\setminus(-\eps,\eps))} \int_{\R}  \frac{y v_k(\xi)}{(x-\xi)^2+y^2} d\xi
= \|v_k\|_{L^\infty(I\setminus(-\eps,\eps))}\to 0, $$
as $k\to \infty$. Moreover, assuming $(x,y)\in \R^{2}_+\setminus B_{2\eps}(0,0)$, we get
$$
\left|\frac{1}{\pi} \int_{-\eps}^\eps \frac{y v_k(\xi)}{(x-\xi)^2+y^2} d\xi \right|\le \frac{1}{\pi} \int_{-\eps}^\eps\frac{y |v_k(\xi)|}{|(x,y)-(\xi,0)|^2} d\xi \le \frac{y}{\pi \eps^2} \|v_k\|_{L^1(I)} \to 0.
$$
Hence $\wt v_k \to 0$ in  $C^0_{\loc}(\R^{2}_+\setminus B_{2\eps}(0,0))$. Finally, since can $\eps$ be arbitrarily small and $\tilde v_k $ is harmonic in $\R^2_+$, we get $\wt v_k \to 0$ in $C^0_{\loc}(\overline{\R^2_+}\setminus\{(0,0)\}) \cap C^1_{\loc}( \R^2_+)$. 
\end{proof}

%
%
%

\subsection{The two main estimates and completion of the proof}

We shall now conclude our contradiction argument by showing the incompatibility of \eqref{ukinfty} with \eqref{convext} and the the definition of $C_\pi$. In this final part of the proof, we will use the precise asymptotic of $\wt G$ near $(0,0)$. Since $\log |(x,y)|$ is the Poisson integral of $\log |x|$ (see Proposition \ref{Uniq}), and since $S\in C(\R)$, \eqref{defS} guarantees the existence of the limit 
$$S_0:=\lim_{(x,y)\to (0,0)} \tilde G(x,y)+\frac{1}{\pi}\log|(x,y)|=\lim_{x\to 0}G(x)+\frac{1}{\pi}\log|x|.$$
In fact, using \eqref{FormulaGI} 
we get $S_0=\frac{\log 2}{\pi}$. More precisely, noting that $S\in C^\infty(I)$, we can write 
\begin{equation}\label{Gprecise}
\tilde G(x,y)=\frac{1}{\pi}\log\frac{1}{|(x,y)|}+ S_0 +h(x,y),
\end{equation}
with $h\in C^\infty(\ov{\R^2_+} \cap B_1(0,0)) \cap C(\ov{\R^2_+})$ and $h(0,0)=0$.

\begin{prop}\label{Cg1} If \eqref{ukinfty} holds, then
$C_\pi \le 2\pi e^{\pi S_0}=4\pi.$
\end{prop}

\begin{proof} For a fixed large $L>0$ and a fixed and small $\delta>0$ set
$$a_k:=\inf_{B_{Lr_k}\cap \R^2_+} \tilde u_k,\quad  b_k:=\sup_{B_\delta\cap \R^2_+} \tilde u_k,\quad \tilde v_k:= (\tilde u_k \wedge a_k)\vee b_k.$$
Recalling that $\|\nabla \tilde u_k\|_{L^2}^2=1$, we have
\begin{equation}\label{stimavk}
\int_{(B_\delta \setminus B_{Lr_k})\cap \R^2_+}|\nabla \tilde v_k|^2dxdy\le 1-\int_{\R^2_+\setminus B_\delta} |\nabla \tilde u_k|^2 dxdy -\int_{\R^2_+\cap B_{Lr_k}} |\nabla \tilde u_k|^2 dxdy
\end{equation}
Clearly the left-hand side bounds
\[\begin{split}
\inf_{\substack{\tilde u|_{\R^2_+\cap \de B_{Lr_k}}=a_k\\ \tilde u|_{\R^2_+\cap \de B_\delta}=b_k}} \int_{(B_\delta \setminus B_{Lr_k})\cap \R^2_+}|\nabla \tilde u|^2dxdy&=\int_{(B_\delta \setminus B_{Lr_k})\cap \R^2_+}|\nabla \tilde \Phi_k|^2dxdy \\
&=\pi\frac{(a_k-b_k)^2}{\log \delta -\log (Lr_k)},
\end{split}\]
where the function $\tilde \Phi_k$ is the unique solution to 
\[
\left\{\begin{array}{ll}
\Delta \tilde \Phi_k = 0 & \text{in } \R^2_+\cap (B_\delta \setminus B_{Lr_k}),\\
\tilde \Phi_k=a_k& \text{on  }\R^2_+\cap \de B_{Lr_k},\\
\tilde \Phi_k=b_k& \text{on  }\R^2_+\cap \de B_{\delta},\\
\frac{\de \tilde \Phi_k}{\de y}=0 & \text{on } \de \R^2_+ \cap (B_{\delta} \setminus B_{Lr_k}),
\end{array}\right.
\]
given explicitly by
$$\tilde \Phi_k=\frac{b_k-a_k}{\log\delta -\log (Lr_k)}\log|(x,y)|+\frac{a_k\log\delta -b_k\log Lr_k}{\log \delta -\log(Lr_k)}.$$
Using Proposition \ref{conveta} we obtain
$$a_k=\mu_k+\frac{-\frac{1}{\pi}\log L+O(L^{-1})+o(1)}{\mu_k},$$
where for fixed $L>0$ we have $o(1)\to0$ as $k\to\infty$, and $|O(L^{-1})|\le \frac{C}{L}$ uniformly for $L$ and $k$ large. Moreover, using Proposition \ref{convgreen} and \eqref{Gprecise}, we obtain
$$b_k=\frac{-\frac{1}{\pi}\log \delta + S_0 +O(\delta)+o(1)}{\mu_k},$$
where for fixed $\delta>0$ we have $o(1)\to0$ as $k\to\infty$, and $|O(\delta)|\le C\delta$ uniformly for $\delta$ small and $k$ large. 

Still with Proposition \ref{conveta} we get
\[
\begin{split}
\lim_{k\to\infty}\mu_k^2\int_{\R^2_+\cap B_{Lr_k}} |\nabla \tilde u_k|^2 dxdy &=\frac{1}{4\pi^2}\int_{\R^2_+\cap B_L}|\nabla \tilde \eta_\infty|^2dxdy\\
&=\frac{1}{\pi}\log\frac{L}{2} +O\bra{\frac{\log L}{L}}.
\end{split}\]
Similarly with Proposition \ref{convgreen} we get
\[
\begin{split}
\liminf_{k\to\infty} \mu_k^2\int_{\R^2_+\setminus B_\delta}|\nabla \tilde u_k|^2dxdy
&\ge \int_{\R^2_+\setminus B_\delta}|\nabla \tilde G|^2dxdy\\
&=\int_{\R^2_+\cap \partial B_\delta} -\frac{\de \tilde G}{\de r} \tilde G d\sigma + \int_{(\R \times \{0\} )\setminus B_\delta} -\frac{\de \tilde G(x,0)}{\de y}  G(x)dx \\
&=\int_{\R^2_+\cap \partial B_\delta} \bra{\frac{1}{\pi \delta}+O(1)}\bra{-\frac{1}{\pi}\log\delta +S_0+O(\delta)} d\sigma \\
&= -\frac{1}{\pi}\log \delta +S_0+O(\delta \log\delta),
\end{split}
\]
where we used the expansion in \eqref{Gprecise} and the boundary conditions
$$
\begin{cases}
\wt G(x,0)= G(x)=0, \quad & \text{ for } x\in \R\setminus I, \\
-\frac{\de \tilde G(x,0)}{\de y}=\hl G(x)=0, \quad  & \text{for }x\in I\setminus \{0\}.
\end{cases}
$$
We then get
\[\begin{split}
\frac{\pi(a_k-b_k)^2}{\log\delta -\log(Lr_k)}&\le 1-\frac{-\frac{1}{\pi}\log \delta +S_0+O(\delta\log \delta)+\frac{1}{\pi}\log\frac{L}{2}+O\bra{\frac{\log L}{L}}}{\mu_k^2},
\end{split}\]
or
\[\begin{split}
\pi(a_k-b_k)^2&= \pi \mu_k^2 -2\log L +O(L^{-1})+2\log \delta -2\pi S_0+O(\delta)+o(1)+\frac{O(\log^2 L +\log^2 \delta)}{\mu_k^2}\\
&\le (\log \delta -\log L +\log (\lambda_k \mu_k^2)+\alpha_k\mu_k^2+\log\alpha_k ) \\
&\quad \times\bra{1-\frac{-\frac{1}{\pi}\log \delta +S_0+O(\delta\log \delta)+\frac{1}{\pi}\log\frac{L}{2}+O\bra{\frac{\log L}{L}}}{\mu_k^2}}\\
&= \log \delta -\log L +\log (\lambda_k \mu_k^2)+\alpha_k\mu_k^2+\log\alpha_k  +\alpha_k\bra{\frac{1}{\pi}\log\delta -S_0- \frac{1}{\pi} \log\frac{L}{2}}\\
&\quad +O(\delta \log\delta)+O\bra{\frac{\log L}{L}}+\frac{O(\log^2\delta)+O(\log^2L)+O(1)}{\mu_k^2}.
\end{split}\]
Rearranging gives
\[\begin{split}
\log\frac{1}{\lambda_k \mu_k^2}&\le  \bra{1-\frac{\alpha_k}{\pi}}(\log L-\log\delta)+(\alpha_k-\pi)\mu_k^2 +(2\pi-\alpha_k)S_0
+\frac{\alpha_k}{\pi}\log 2 +\log\alpha_k\\
&\quad +O(\delta \log \delta )+ O\bra{\frac{\log L}{L}}+o(1),
\end{split}\]
with $o(1)\to 0$ as $k\to \infty$.
Then, recalling that $\alpha_k\uparrow \pi$, letting $k\to \infty$ first and then $L\to \infty$, $\delta\to 0$, we obtain
$$\limsup_{k\to\infty}\log  \frac{1}{\lambda_k \mu_k^2} \le \pi S_0+ \log(2\pi)=\log(4\pi),$$
and using Proposition \ref{stimaalto} we conclude.
\end{proof}

\begin{prop}\label{Cg2}
There exists a function $u\in \tilde H^{\frac{1}{2},2}(I)$ with $\|\ql u\|_{L^2(\R)}\le 1$ such that
$$\int_I \bra{e^{\pi u^2}-1}dx> 2\pi e^{\pi S_0}=4\pi.$$
\end{prop}

\begin{proof}  For $\ve>0$ choose $L=L(\ve)>0$ such that as $\ve\to 0$ we have $L\to\infty$ and $L\ve\to 0$. Fix
$$\Gamma_{L\ve}:=\left\{(x,y)\in \R^2_+: \tilde G(x,y)=\gamma_{L\ve}:=\min_{\R^2_+\cap\de B_{L\ve}} \tilde G\right\}, $$
and
$$\Omega_{L\ve}:=\left\{ (x,y)\in \R^2_+: \tilde G(x,y)>\gamma_{L\ve} \right\}.$$
By the maximum principle we have $\R^2_+\cap B_{L\ve}\subset \Omega_{L\ve}$. Indeed, $\tilde G$ is harmonic in $\R^2_+$, $\tilde G\ge \gamma_{L\eps}$ on $\partial (\R^2_+\cap B_{L\eps}) \setminus \{(0,0)\}$,  and $\tilde G\to +\infty $ as $(x,y)\to (0,0)$.
Notice also that \eqref{Gprecise} gives
\begin{equation}\label{ExpGamma}
\gamma_{L\ve} =-\frac{1}{\pi}\log(L\ve)+S_0+O(L\ve).
\end{equation}

For some constants $B$ and $c$ to be fixed we set
\[
U_\ve(x,y):=\left\{
\begin{array}{ll}
\displaystyle c- {\frac{\log\bra{\frac{x^2}{\eps^2}+(1+\frac{y}{\eps})^2} +2B }{2\pi c}}& \text{for } (x,y)\in \R^2_+\cap B_{L\ve}(0,-\eps) \\
\displaystyle\rule{0cm}{0.5cm} \frac{\gamma_{L\ve}}{c} &\text{for }(x,y)\in \Omega_{L\ve}\setminus  B_{L\ve}(0,-\eps)\\
\displaystyle\rule{0cm}{0.8cm}  \frac{\tilde G(x,y)}{c} &\text{for }(x,y)\in  \R^2_+ \setminus \Omega_{L\ve}.
\end{array}
\right.
\]
Observe that $\R^2_+\cap B_{L\eps}(0,-\eps) \subseteq \R^2_+\cap B_{L\eps} \subseteq \Omega_{L\eps}$.
To have continuity on $\R^2_+\cap \de B_{L\ve}(0,-\ve)$ we impose
$$\frac{-\log L^2 -2B }{2\pi c}+ c=\frac{\gamma_{L\ve}}{c} 
$$
which, together with \eqref{ExpGamma}, gives the relation
\begin{equation}\label{contoB}
B=\pi c^2+\log\ve-\pi S_0+O(L\ve).
\end{equation}
Moreover
\[\begin{split}
\int_{ \R^2_+ \cap B_{L\ve}(0,-\ve)} |\nabla U_\ve|^2dxdy&=\frac{1}{4\pi^2 c^2}\int_{\R^2_+\cap B_L(0,-1)}|\nabla \log(x^2+(1+y)^2)|^2dxdy\\
&=\frac{\frac{1}{\pi}\log\bra{\frac{L}{2}}+O\bra{\frac{\log L}{L}}}{c^2},
\end{split}\]
and
\[\begin{split}
\int_{\R^2_+\setminus\Omega_{L\ve}}|\nabla U_\ve|^2 dxdy&=\frac{1}{c^2}\int_{\R^2_+\setminus\Omega_{L\ve}}|\nabla\tilde G|^2 dxdy\\
&=\frac{1}{c^2}\int_{\R^2_+\cap \de \Omega_{L\ve}}\frac{\de \tilde G}{\de \nu} \tilde G d\sigma-\frac{1}{c^2}\underbrace{\int_{(\R\times\{0\})\setminus \bar \Omega_{L\ve}} \frac{\de \tilde G}{\de y} \tilde G dx}_{=0}\\
&=\frac{\frac{1}{\pi}\log\bra{\frac{1}{L\ve}} +S_0+O(L\ve \log(L\ve))}{c^2},
\end{split}\]
where the last equality follows from \eqref{Gprecise}. We now impose $\|\nabla U_\ve\|_{L^2(\R^2_+)}=1$, obtaining
\begin{equation}\label{contoc}
-\log \ve -\log 2+\pi S_0+O(L\ve \log(L\ve))+ O\bra{\frac{\log L}{L}}=\pi c^2,
\end{equation}
which, together with \eqref{contoB}, implies
\begin{equation}\label{contoB2}B=-\log 2+ O(L\ve \log(L\ve))+ O\bra{\frac{\log L}{L}}.
\end{equation}
Let now $I_{L,\eps}^1=(-\eps \sqrt{L^2-1},\eps \sqrt{L^2-1})$ and $I_{L\eps}^2$  be the disjoint sub-intervals of $I$ obtained by intersecting $I\times \{0\}$ respectively with $B_{L\eps}(0,-\eps)$ and $\ov{\R^2_+ \setminus \Omega_{L\eps}}$. Then, for $u_\ve(x):=U_\ve(x,0)$,   using a change of variables and \eqref{contoc}-\eqref{contoB2} we get 
\[\begin{split}
\int_{I_{L,\eps}^1} e^{\pi u_\ve^2}dx
&=\ve  \int_{-\sqrt{L^2-1}}^{\sqrt{L^2-1}}\exp\bra{\pi \bra{c-\frac{\log(1+x^2)+2B}{2\pi c}}^2}dx\\
&>\ve e^{\pi c^2-2B}\int_{-\sqrt{L^2-1}}^{\sqrt{L^2-1}}\frac{1}{1+x^2}dx\\
&=2e^{\pi S_0+ O(L\ve \log(L\ve))+ O\bra{\frac{\log L}{L}}} \pi\bra{1+O\bra{\frac 1L}}\\
&=2\pi e^{\pi S_0}+O(L\ve \log(L\ve))+ O\bra{\frac{\log L}{L}}.
\end{split}\]
Moreover
\[
\int_{I_{L\eps}^2} \bra{e^{\pi u_\ve^2}-1}dx  \ge \int_{I_{L\eps}^2} \pi u_\ve^2dx
=\frac{1}{c^2} \int_{I_{L\eps}^2} \pi G^2dx=:\frac{\nu_{L\ve}}{c^2},
\]
with
$$\nu_{L\ve}>\nu_{\frac{1}{2}}>0,\quad \text{for } L\ve<\frac{1}{2}.$$
Now observe  that $c^2 = -\frac{\log \eps}{\pi} +O(1)$ by \eqref{contoc}, and choose $L=\log^2\ve$  to obtain
$$O(L\ve \log(L\ve))+ O\bra{\frac{\log L}{L}}= O\bra{\frac{\log\log\ve}{\log^2\ve}}= o\bra{\frac{1}{c^2}},$$
so that
\[\begin{split}
\int_{I} \bra{e^{\pi u_\ve^2}-1}dx &\ge 2\pi e^{\pi S_0} +\frac{\nu_{\frac12}}{c^2}+o\bra{\frac{1}{c^2}}>2\pi e^{\pi S_0}
\end{split}\]
for $\ve$ small enough.

Finally notice that
\[
\|\ql u_\ve \|_{L^2(\R)}^2=\int_{\R^2_+}|\nabla \tilde u_\ve|^2dxdy\le \int_{\R^2_+}|\nabla U_\ve|^2dxdy \le 1,
\]
since the Poisson  extension $\tilde u_\ve$ minimizes the Dirichlet energy among extensions with finite energy.
\end{proof}

\section{Proof of Theorem \ref{compact}}
Let $u_k\in H\cap C^\infty(\R)$ be a sequence of positive {even and decreasing} solutions to \eqref{eqR} {satisfying} the energy bound \eqref{energybound} {and with} $\lambda_k\to \lambda_\infty \ge 0$ as $k\to \infty$.

First we show that case (i) holds when $\mu_k\le C$.

\begin{lemma}\label{lemma3.1}
If $\mu_k\le C$ then (i) holds. 
\end{lemma}
\begin{proof}
By assumption we know that  $u_k$ and $f_k:=(-\Delta)^\frac{1}{2} u_k=\lambda_k u_k e^{\alpha_k u_k^2}-u_k$ are uniformly bounded in $L^\infty(\R)$. Then, by elliptic estimates and a bootstrap argument, we can find $u_\infty\in C^{\infty}(\R)$ such that up to a subsequence $u_k\to u_\infty$  in $C^\ell_{\loc}(\R)$ for every $\ell\ge 0$. To prove that $u_{\infty}$ satisfies \eqref{eqinftyR}, note that $f_k\to f_\infty:= \lambda_\infty u_\infty e^{\pi u_\infty^2}-u_{\infty}$ locally uniformly on $\R$ and set $M= \sup_{k} (\|f_k\|_{L^\infty(\R)}+ \mu_k)$. For any $\varphi \in \mathcal{S}(\R)$ (the Schwarz space of rapidly decreasing functions) and any $R>0$, we have that  
\[
\begin{split}
\int_{\R} |f_k - f_\infty||\varphi| dx  & \le  \|f_k-f_\infty\|_{L^\infty((-R,R))}\int_{-R}^R |\varphi| + 2M \|\varphi\|_{L^1{((-R,R)^c)}}\\
& \stackrel{k\to +\infty}{\rightarrow} M  \|\varphi\|_{L^1((-R,R)^c)} \\ & \stackrel{R\to +\infty}{\rightarrow} 0.
\end{split}
\]
Similarly, recalling that $(-\Delta)^\frac{1}{2} \ph$ has quadratic decay at infinity {(see e.g. \cite[Prop. 2.1]{Hyd})}, we get
\[\begin{split}
\int_{\R} |u_k- u_\infty| |(-\Delta)^\frac{1}{2} \varphi| dx  & \le \|(-\Delta)^\frac{1}{2} \varphi \|_{L^\infty((-R,R))} \|u_k-u_\infty\|_{L^1((-R,R))} \\
&\quad + {C} \int_{(-R,R)^c}  \frac{|u_k(x)-u_\infty(x)|}{|x|^2} dx  \\
&\le  \|(-\Delta)^\frac{1}{2} \varphi \|_{L^\infty((-R,R))} \|u_k-u_\infty\|_{L^1((-R,R))} + 2{C}M \int_{(-R,R)^c} \frac{dx}{x^2}\\
& \stackrel{k,R\to +\infty}{\rightarrow} 0.
\end{split}
\]
Hence $u$ is a weak solution of \eqref{eqinftyR}.  
\end{proof}


From now on we will assume that $\mu_k\to +\infty$ and prove that (ii) of Theorem \ref{compact} holds.

\begin{lemma}\label{localconv}
Let $\eta_k$ be defined as in Theorem \ref{compact}. Then $\eta_k$ is bounded in $C^{0,\alpha}_{\loc}(\R)$ for  $\alpha\in (0,1)$. 
\end{lemma}
\begin{proof}
Note that 
\[\begin{split}
r_k  \mu_k^2  = \frac{1}{\alpha_k \lambda_k e^{\alpha_k \mu_k^2}}   & =  \frac{1}{\alpha_k \|u_k\|_{H}^2 e^{\alpha_k\mu_k^2}} \int_{\R} u_k^2 e^{\alpha_ku_k^2}dx \\
& \le C \frac{1}{\alpha_k \|u_k\|_{H}^2 e^{\frac{\alpha_k}{2} \mu_k^2}} \int_{\R} u_k^2 e^{\frac{\alpha_k}{2}u_k^2}dx \\ & \le C \frac{\|u_k\|_{L^4}^2 \sqrt{D_{\alpha_k}}}{\alpha_k \|u_k\|_{H}^2 e^{\frac{\alpha_k}{2} \mu_k^2} } \\ &\le C  \frac{\sqrt{D_\pi}}{\alpha_k e^{\frac{\alpha_k}{2} \mu_k^2}}\to 0.  
\end{split}
\]
Moreover we have that
$$
(-\Delta)^\frac{1}{2}\eta_k = 2\frac{u_k(r_k\cdot)}{\mu_k}e^{\alpha_k u_k^2(r_k\cdot)-\alpha_k \mu_k^2} - 2\alpha_k r_k \mu_k^2  \frac{u_k(r_k \cdot)}{\mu_k}
$$
is bounded in $L^\infty$. Since $\eta_k \le 0$, and $\eta_k(0)=0$ this implies that $\eta_k$ is bounded in $L^\infty_{\loc}(\R)$ and then in $C^\alpha_{\loc}(\R)$ for any $\alpha\in (0,1)$. 
\end{proof}

The bound of Lemma \ref{localconv} implies that, up to a subsequence $\eta_k\to \eta_\infty$ in $C^{0,\alpha}_{\loc}(\R)$ for some function $\eta_\infty$. However, it does not provide a limit equation for $\eta_\infty$. In order to prove that $\eta_\infty$ solves 
$$
(-\Delta)^\frac{1}{2}\eta_\infty = 2 e^{\eta_\infty}
$$
we will prove that that $\eta_k$ is bounded in $L_s(\R)$ for any $s>0$. This bound can be obtained thanks to the commutator estimates proved in \cite{MMS}. Part of the argument must be modified since the ${u_k}'s$ are not compactly supported. We start by recalling the following technical lemma, which is a consequence of the estimates in \cite{MMS}.

\begin{lemma}\label{Strange Lemma} 
For any $s\in (0,1)$, there exists a constant $C=C(s)$ such that, for any  $\ph,\psi \in C^\infty_c(\R^n)$, 
 $\rho \in \R^+$,   we have 
\[
\|\ph (-\Delta)^\frac{s}{2}\psi \|_{L^{(\frac{1}{s},\infty)}((-\rho,\rho))} \le  C \left( E_1(\ph,\psi )+ E_{2,2\rho}(\ph,\psi) \right),
\] 
where
\[
\begin{split}
&E_1(\ph,\psi)= \|(-\Delta)^\frac{1}{4}\ph\|_{L^2(\R)} \|(-\Delta)^\frac{1}{4}\psi\|_{L^2(\R)}  \\ 
&E_{2,\rho}(\ph,\psi)= \|(-\Delta)^\frac{1}{4}\ph\|_{L^2(\R)}  \|(-\Delta)^\frac{1}{2}\psi\|_{L\log^\frac{1}{2} L (-\rho,\rho)}, \\ 
\end{split}
\]
\end{lemma}
\begin{proof}
Let $\theta \in C^\infty_c((-2,2))$ be a cut-off function such that $\theta \equiv 1$ on $(-1,1)$ and $0\le \theta\le 1$. Let us denote $\theta_\rho =\theta(\frac{\cdot}{\rho})$. Let us also introduce the Riesz operators
$$I_{1-s} u:= \kappa_s|\cdot |^{-s}*u\quad \text{for }s\in (0,1),$$
where the constant $\kappa_s$ is defined by the identity $ \widehat{ \kappa_s |\cdot|^{-s}} = {|\cdot|^{s-1}}$. With this definition $I_{1-s}$ is the inverse of $(-\Delta)^{\frac{1-s}{2}}$. Then we can split 
\[
\begin{split}
\ph (-\Delta)^\frac{s}{2} \psi  &= \ph I_{1-s} (-\Delta)^\frac{1}{2}\psi  \\
& =   \ph I_{1-s} \bra{\theta_{2\rho}(-\Delta)^\frac{1}{2}\psi} +   \ph I_{1-s} \bra{(1-\theta_{2\rho})(-\Delta)^\frac{1}{2}\psi} \\
& =  \ph I_{1-s} \bra{\theta_{{2\rho}}(-\Delta)^\frac{1}{2}\psi} + [\ph, I_{1-s}] \bra{(1-\theta_{2\rho})(-\Delta)^\frac{1}{2}\psi} +  I_{1-s} \bra{(1-\theta_{2\rho}) \ph (-\Delta)^\frac{1}{2}\psi},
\end{split}
\]
where we use the commutator notation $[u,I_{1-s}](v)= u I_{1-s} v - I_{1-s}(uv)$ for any $u,v\in C^\infty_c(\R)$. Applying respectively Proposition 3.2, Proposition 3.4 and Proposition A.3. in \cite{MMS}, we get that
\[
\begin{split}
\|  \ph I_{1-s} \bra{\theta_{2\rho}(-\Delta)^\frac{1}{2}\psi} \|_{L^{(\frac{1}{s},\infty)}(-\rho,\rho)} &=  \|  I_\frac{1}{2}\bra{(-\Delta)^\frac{1}{4}\ph} I_{1-s} \bra{\theta_{2\rho}(-\Delta)^\frac{1}{2}\psi} \|_{L^{(\frac{1}{s},\infty)}(-\rho,\rho)} \\
& \le  C \|(-\Delta)^\frac{1}{4}\ph\|_{L^2(\R)}  \|(-\Delta)^\frac{1}{2}\psi\|_{L\log^\frac{1}{2} L(-2\rho,2\rho)} \\
& = C E_{2,2\rho}(\ph,\psi),
\end{split}
\]
that
\[
\begin{split}
\|  [\ph, I_{1-s}] \bra{(1-\theta_{2\rho})(-\Delta)^\frac{1}{2}\psi} \|_{L^{(\frac{1}{s},\infty)}(-\rho,\rho)} & =   \|  [\ph, I_{1-s}] \bra{(1-\theta_{2\rho})(-\Delta)^\frac{1}{4} (-\Delta)^\frac{1}{4} \psi } \|_{L^{(\frac{1}{s},\infty)}(-\rho,\rho)}  \\
& \le C \|(-\Delta)^\frac{1}{4}\ph\|_{L^2(\R)} \|(-\Delta)^\frac{1}{s}\psi\|_{L^2(\R)}\\
& =  C E_1(\ph,\psi),
\end{split}
\]
and that
\[
\begin{split}
\| I_{1-s} \bra{(1-\theta_{2\rho}) \ph (-\Delta)^\frac{1}{2}\psi} \|_{L^{(\frac{1}{s},\infty)}(-\rho,\rho)} & \le  \| I_{1-s} \bra{ \ph (-\Delta)^\frac{1}{2}\psi} \|_{L^{(\frac{1}{s},\infty)}(\R)}  \\ &  \le C \|\ph (-\Delta)^\frac{1}{2}\psi\|_{L^1(\R)} \\
& = C \|(-\Delta)^\frac{1}{4} \ph (-\Delta)^\frac{1}{4}\psi\|_{L^1(\R)} \\
&\le C E_1(\ph,\psi).
\end{split}
\] 
\end{proof}

As a consequence of Lemma \ref{Strange Lemma} we obtain the following crucial estimate. 
 
%
 
\begin{lemma}\label{Lemma crucial}
For any $s\in (0,1)$ there exists a constant $C=C(s)$ such that  
$$
\int_{(-\rho,\rho)} | u (-\Delta)^\frac{s}{2} u| dx \le C  \rho^{1-s} (E_1(u,u)+E_{2,2\rho}(u,u))
$$
for any $\rho>0$,  and $u\in H \cap C^\infty(\R)$. Here $E_1$ and $E_{2,2\rho}$ are defined as in Lemma \ref{Strange Lemma}. 
\end{lemma}
\begin{proof} By the H\"older inequality for Lorentz spaces (see e.g. \cite[Theorem 3.5]{ON}), we have 
\begin{equation}\label{uDeltas2u}
\begin{split}
\|u(-\Delta)^\frac{s}{2}u\|_{L^1(-\rho,\rho)} &\le  \| \chi_{(-\rho,\rho)}\|_{L^{(\frac{1}{1-s},1)}(\R)}\|u (-\Delta)^\frac{s}{2}u\|_{L^{(\frac{1}{s},\infty)}(-\rho,\rho)} \\
& \le C \rho^{1-s}\|u (-\Delta)^\frac{s}{2}u\|_{L^{(\frac{1}{s},\infty)}(-\rho,\rho)}.
\end{split} 
\end{equation}
%
We shall bound the RHS of \eqref{uDeltas2u} by approximating $u$ with compactly supported functions and applying Lemma \ref{Strange Lemma}. To this purpose, we take a sequence of cut-off function $(\tau_j)_{j\in \N} \subseteq  C^\infty_c(\R)$ such that $\tau_j (x)=1$ for $|x|\le j$, $\tau_j(x)=0$ for $|x|\ge j+1$, $0\le \tau_j \le 1$ and $|\tau_j'|\le 2$. We define $u_j:=\tau_j u$. We claim that
\begin{equation}\label{convinH}
u_j\rightarrow u \quad \text{ in } H^{\frac{1}{2},2}(\R) \cap L^q(\R),\; q\in (2,\infty)
\end{equation}
and
\begin{equation}\label{convlap}
(-\Delta)^\frac{s}{2} u_j \rightarrow (-\Delta)^\frac{s}{2} u\quad  \text{ in } L^\infty_{\loc}(\R). 
\end{equation}
The first claim is proved in \cite[Lemma 12]{FSV}.   We shall prove the second claim.  Set $v_j=u_j -u$. Then, for any fixed $R_0>0$ and $x\in (-R_0,R_0)$, if $j>2R_0$ we have
$$ 
|(-\Delta)^\frac{s}{2}v_j| \le K_{s} \int_{\R\setminus (-j,j)} \frac{|v_j(y)|}{|x-y|^{1+s}} dy  \le 2^{1+s} K_s \int_{\R\setminus (-j,j)} \frac{|u(y)|}{|y|^{1+s}} dy \le C \|u\|_{L^2(\R)} j^{-1-2s}. 
$$ 
with $C$ depending only on $s$.  As $j\to \infty$, we get \eqref{convlap}.
%

Now, By Lemma \ref{Strange Lemma}, we know that, for any $j$,
\begin{equation}\label{j}
\|u_j (-\Delta)^\frac{s}{2} u_j\|_{L^{(\frac{1}{s},\infty)}(-\rho,\rho)} \le C( E_1(u_j,u_j)+E_{2,\rho}(u_j,u_j)),
\end{equation}
where $C$ depends only on $s$. Clearly, \eqref{convinH}  yields 
$$
E_1(u_j,u_j)\rightarrow E_1(u,u).
$$  
Moreover
$$
E_{2,2\rho}(u_j,u_j) = E_{2,2\rho}(u,u),\quad \text{for } j\ge 2\rho. 
$$ 
Finally, \eqref{convinH} and \eqref{convlap} imply that $u_j (-\Delta)^\frac{s}{2} u_j\rightarrow u (-\Delta)^\frac{s}{2} u$ in $L^q_{\loc}(\R)$ for every $q\in [1,\infty)$, and therefore in $L^{(\frac{1}{s},\infty)}(-\rho,\rho)$.  Then, passing to the limit in \eqref{j} we get  
$$
\|u (-\Delta)^\frac{s}{2} u \|_{L^{(\frac{1}{s},\infty)}(-\rho,\rho)} \le C (E_1(u,u)+E_{2,2\rho}(u,u)),
$$
and together with \eqref{uDeltas2u} we conclude.
\end{proof}

We can now apply Lemma \ref{Lemma crucial} to $u_k$. After scaling, we get the following bound on $\eta_k$.

\begin{lemma}\label{lemmaDeltas2}
For any $s\in (0,1)$, there exists a constant $C=C(s)>0$ such that 
$$
\int_{-R}^R |(-\Delta)^\frac{s}{2}\eta_k| dx \le C R^{1-s}, \quad \text{ for any } R>0  \text{ and } k\ge k_0(R).
$$
\end{lemma}
\begin{proof}
First we observe that  $f_k:=(-\Delta)^\frac{1}{2} u_k=\lambda_k u_k e^{\alpha_k  u_k^2} -u_k$  is bounded in $L\log^\frac{1}{2}L_{\loc}(\R)$. Indeed, we have
$$
\log^\frac{1}{2}(2+ |f_k|) \le C(1+u_k),
$$
so that
$$
|f_k|\log^\frac{1}{2}(2+ |f_k|)  \le C |f_k|\left( 1+u_k\right) = O(|f_k|u_k+1).
$$
Since $|f_k|u_k$ is bounded in $L^1(\R)$ by \eqref{eqR} and  \eqref{energybound},  we get that $f_k$ is bounded in $L \log^\frac{1}{2} L_{\loc}(\R)$. 

Then Lemma \ref{Lemma crucial} and \eqref{energybound} imply the existence of $C=C(s)$ such that 
$$
\int_{-\rho}^\rho | u_k (-\Delta)^\frac{s}{2} u_k| dx \le C  \rho^{1-s}, \quad \rho \in (0,1).
$$
For any $R>0$, we can apply this with $\rho = R r_k$ and rewrite it in terms of $\eta_k$. Then, we obtain
$$
\int_{-R}^R \bra{1+\frac{\eta_k}{\mu_k^2}} |(-\Delta)^\frac{s}{2} \eta_k| \le C R^{1-s}.  
$$ 
Since, by Lemma \ref{localconv}, $\eta_k$ is locally bounded, if $k$ is sufficiently large we get $1+\frac{\eta_k}{\mu_k^2}\ge \frac{1}{2}$ and the proof is complete. 
\end{proof}

\begin{lemma}\label{LemmaLs}
The sequence $(\eta_k)$ is bounded in $L_s(\R)$ for any $s>0$.  
\end{lemma}
\begin{proof}  
It is sufficient to prove the statement for $s\in (0,1/2)$. Since $\eta_k \le 0$, Lemma \ref{lemmaDeltas2} gives
\[\begin{split}
C  & \ge \frac{1}{K_s}\int_{-1}^1 |(-\Delta)^s \eta_k| dx \\ 
& \ge  \left|\int_{-1}^1 \int_{\R} \frac{\eta_k(x)-\eta_k(y)}{|x-y|^{1+2s}}dydx \right|\\ 
&\ge    \underbrace{\int_{-1}^1 \int_{-2}^2 \frac{\eta_k(x)-\eta_k(y)}{|x-y|^{1+2s}}dydx}_{=:I_1} +  \underbrace{\int_{-1}^1 \int_{(-2,2)^c} \frac{\eta_k(x)}{|x-y|^{1+2s}}dydx}_{=:I_2} +   \underbrace{\int_{-1}^1 \int_{(-2,2)^c} \frac{-\eta_k(y)}{|x-y|^{1+2s}}dydx}_{=:I_3}. 
\end{split}\]
Take $2s <\alpha<1$. Since $\eta_k$ is bounded in $C^\alpha_{\loc}(\R)$ by Lemma \ref{localconv}, we have that 
$$
|I_1| \le C \int_{-1}^1\int_{-2}^2 \frac{dy dx}{|x-y|^{1+2s-\alpha}} \le C \int_{-3}^3 \frac{dz}{|z|^{1+2s-\alpha}} = C. 
$$
Similarly 
$$
|I_2|\le \int_{-1}^1 |\eta_k(x)| \int_{(x-1,x+1)^c} \frac{1}{|x-y|^{1+2s}} dydx \le C.  
$$
Therefore, we obtain that
$$
I_3=\int_{-1}^1\int_{(-2,2)^c} \frac{|\eta_k(y)|}{|x-y|^{1+2s}} dy dx\le C.  
$$
But for $x\in (-1,1)$ and $y\notin (-2,2)$ we have  $|x-y|\le |y|+|x|\le 2|y| \le 2(1+|y|^{1+2s})^{\frac{1}{1+2s}}$.  Hence 
$$
I_3=  \int_{-1}^1\int_{(-2,2)^c} \frac{|\eta_k(y)|}{|x-y|^{1+2s}} dy dx \ge \frac{1}{2^{2s}} \int_{(-2,2)^c} \frac{|\eta_k(y)|}{1+|y|^{1+2s}}dy.  
$$

This and Lemma \ref{localconv} imply that $\eta_k$ is bounded in $L_s(\R)$. 
\end{proof}
\medskip

\noindent\emph{Proof of Theorem \ref{compact} (completed).} By Lemma \ref{localconv}, up to a subsequence we can assume that $\eta_k \to \eta_\infty$ in $C^\alpha
_{\loc}(\R)$ for any $\alpha\in (0,1)$, with $\eta_\infty \in C^\alpha_{\loc}(\R)$. Let us denote 
$$
f_k:= (-\Delta)^\frac{1}{2} \eta_k =  2\bra{1+\frac{\eta_k}{2\alpha_k \mu_k^2}}e^{\eta_k+\frac{\eta_k^2}{4\alpha_k \mu_k^2}} - 2 r_k \alpha_k \mu_k^2 \bra{1+\frac{\eta_k}{2\alpha_k \mu_k^2}}.
$$
As observed in the proof of Lemma \ref{localconv}, we have $r_k\mu_k^2\to 0$ as $k\to \infty$ and thus  $f_k\to 2e^{\eta_\infty}$ locally uniformly on $\R$. Moreover $f_k$ is bounded in $L^\infty(\R)$. Then,  for any Schwarz function $\ph \in \mathcal{S}(\R)$ we have 
$$
\int_{\R} |f_k -2 e^{\eta_\infty}| |\ph| dx\le o(1)\int_{(-R,R)} |\ph|dx + (\|f_k\|_{L^\infty(\R)} + \|2e^{\eta_\infty}\|_{L^\infty(\R)})\int_{(-R,R)^c} |\ph| dx\rightarrow 0 
$$
as $k,R\to+\infty$. On the other hand, we know by Lemma \ref{LemmaLs} that $\eta_k$ is bounded in $L_s(\R)$ and, consequently, $\eta_\infty\in L_s(\R)$, $s>0$. In particular, for $s  \in (0,\frac{1}{2})$, letting $k\to\infty$ first, and then $R\to\infty$ we get
\[\begin{split}
\int_{\R}|\eta_k-\eta_\infty| &|(-\Delta)^\frac{1}{2} \ph|  dx \\ & \le \|(-\Delta)^\frac{1}{2}\ph \|_{L^\infty(-R,R)} \|\eta_k-\eta_\infty\|_{L^1(-R,R)} +C \int_{(-R,R)^c} \frac{|\eta_k(x)-\eta_\infty(x)|}{|x|^2} dx \\ &\le C  \|\eta_k-\eta_\infty\|_{L^1(-R,R)} + C R^{2s-1} (\|\eta_k\|_{L_s(\R)}+ \|\eta_\infty\|_{L_s(\R)})\rightarrow 0. 
\end{split}\]
Then $\eta_\infty$ is a weak solution $(-\Delta)^\frac{1}{2} \eta_\infty= 2 e^{\eta_\infty}$ and $\eta_\infty\in L_s(\R)$ for any $s$. Moreover, repeating the argument of Corollary \ref{uinfty0} and using \eqref{energybound}, we get
\begin{equation}\label{energysmall}
\frac{1}{\pi} \int_{-R}^{R} e^{\eta_\infty} d\xi = \lim_{k\to \infty} \int_{-R r_k}^{R rk} \lambda_k u_k^2 e^{\alpha_k u_k^2} dx \le \limsup_{k\to \infty} \|u_k\|_{H}^2 = \Lambda,
\end{equation}
which implies $e^{\eta_\infty}\in L^1(\R)$. Then $\eta_\infty(x)= -\log(1+x^2)$, see e.g. \cite[Theorem 1.8]{DMR}. 

To complete the proof, we shall study the properties of the weak limit $u_\infty$ of $u_k$ in $H$. First, we show that $u_\infty$ is a weak solution of \eqref{eqinftyR}. Let us denote 
$$g_k := \lambda_k u_k e^{\alpha_k u_k^2},\quad g_\infty := \lambda_\infty u_\infty e^{\pi u_\infty^2}.$$
Take any function $\ph\in \mathcal{S}(\R)$. On the one hand, since $(-\Delta)^\frac{1}{2} \ph\in L^2(\R)$ and $u_k \rw u_\infty$ weakly in $L^2(\R)$, we have
$$
\int_{\R} (u_k-u_\infty) (-\Delta)^\frac{1}{2}\ph \, dx +  \int_{\R} (u_k-u_\infty) \ph \, dx\to 0,
$$
as $k\to \infty$. On the other hand,  for any large $t>0$ we get
$$
\int_{\R} | g_k- g_\infty| |\ph| dx \le \int_{\{u_k\le t\}} |  g_k-g_\infty||\ph| dx+ \frac{\|\ph\|_{L^\infty(\R)}}{t}  \int_{\R} u_k (g_k+ g_\infty)   dx = o(1)+O(t^{-1})\to 0
$$
as $k,t\to \infty$, where we used that $g_\infty\in L^2(\R)$ by Theorem \ref{MT3} (see e.g. Lemma 2.3 of \cite{IMM}) together with the dominated convergence theorem and the bounds $\|u_k g_k\|_{L^1(\R)}\le \Lambda$ and $\|u_k\|_{L^2(\R)}\le \Lambda$. 
Then, $u_\infty$ is a weak solution of \eqref{eqinftyR}. 

Now, observe that
$$
\|u_k\|_{H}^2 = \int_{\R} g_k u_k dx = \int_{-Rr_k}^{Rr_k} g_k u_k dx + \int_{\R \setminus (-Rr_k,Rr_k)} g_k u_k dx 
$$ 
with 
$$
\lim_{k\to \infty} {\int_{-Rr_k}^{Rr_k}} u_k g_k \, dx =  \frac{1}{\pi} \int_{-R}^R e^{\eta_\infty} dx \to 1
$$
as $R\to \infty$, and 
$$
\liminf_{k\to \infty} \int_{\R \setminus (-Rr_k,Rr_k)} g_k u_k dx = \int_{\R} g_\infty u_\infty dx = \|u_\infty\|_{H}^2,
$$
for any $R>1$, by Fatou's lemma. Thus we conclude that
$$
\|u_k\|_{H}^2 \ge \|u_\infty\|_H^2 + 1. 
$$
{Finally, to prove that $u_k\to u_\infty$ in $C^\ell_{\loc}(\R\setminus\{0\})$ for every $\ell\ge 0$, we use the monotonicity of $u_k$, which implies that $u_k$ is locally bounded away from $0$, hence we can conclude by elliptic estimates, as in Lemma \ref{lemma3.1}.}


\section{Proof of Theorem \ref{extR}}
Let us denote 
$$
E_\alpha(u)=\int_{\R} (e^{\alpha u^2}-1) dx,\quad D_\alpha:=\sup_{u\in H: \|u\|_H\le 1}E_\alpha(u).
$$         

The proof of Theorem \ref{extR} is organized as follows. First, we prove that $D_\alpha$ is attained for $\alpha\in (0,\pi)$ sufficiently close to $\pi$. Then, we fix a sequence $(\alpha_k)_{k\in N}$ such that $\alpha_k \nearrow \pi$ as $k\to +\infty$, and for any large $k$ we take a positive extremal $u_k \in H$ for  $D_{\alpha_k}$. With a contradiction argument similar to the one of Section \ref{SecExtI}, we show that $\mu_k:= \sup_{\R} u_k\le C$. Finally, we show that $u_k \to u_\infty$ in $L^\infty_{\loc}(\R) \cap L^2(\R)$, where $u_\infty$ is a maximizer for $D_{\pi}$.

\subsection{Subcritical extremals: Ruling out vanishing}
The following lemma describes the effect of the lack of compactness of the embedding ${H\subseteq L^2(\R)}$ on $E_\alpha$, and holds uniformly for $\alpha\in [0,\pi]$.

\begin{lemma}\label{LemmaVanishing}
Let $(\alpha_k)\subseteq [0,\pi]$ and $(u_k)\subseteq  H $ be two sequences such that:
\begin{enumerate}
\item $\alpha_k\to \alpha_\infty \in [0,\pi]$ as $k\to \infty$.
\item $\|u_k\|_{H}\le 1$, $u_k \rw u_{\infty}$ weakly in $H$, $u_k\to u_\infty$ a.e. in $\R$, and $e^{\alpha_k u_k^2}\to e^{\alpha_\infty u_\infty^2}$ in $L^1_{\loc}(\R)$ as $k\to \infty$.
\item The $u_k$'s are even and monotone decreasing i.e. $u_k(-x)=u_k(x)\ge u_k(y)$ for $0\le x \le y$.
\end{enumerate} 
Then we have
$$
 E_{\alpha_k} (u_k) =E_{\alpha_\infty}(u_\infty) + \alpha_\infty \left(\| u_k\|_{L^2(\R)}^2 - \|u_\infty\|_{L^2(\R)}^2 \right) +o(1),
$$
as $k\to \infty$. 
\end{lemma}
\begin{proof}
Since $u_k$ is even and decreasing, we know that 
\begin{equation}\label{decay}
u_k(x)^2 \le \frac{\|u_k\|_{L^2(\R)}^2}{2|x|} \le \frac{1}{2|x|},
\end{equation}
for any $x\in \R\setminus \{0\}$. In particular, there exists a constant $C>0$, such that 
$$
e^{\alpha_k u_k^2(x)} - 1-\alpha_k u_k^2(x) \le C |x|^{-4},
$$
for $|x|\ge 1$.  Applying the dominated convergence theorem for $|x|\ge 1$,  using the assumption  that $e^{\alpha_k u_k^2} \to e^{\alpha_\infty u_\infty^2}$ in $L^1_{\loc}(\R)$, and recalling that $(u_k)$ is precompact in $L^1_{\loc}(\R)$,  we find that 
$$
\int_{\R} (e^{\alpha_k u_k^2} - 1-\alpha_k u_k^2 ) dx \to \int_{\R} (e^{\alpha_\infty u_\infty^2} - 1-\alpha_\infty u_\infty^2)dx,
$$
and the Lemma  follows. 
\end{proof}

\begin{lemma}\label{novan}
Take $\alpha\in (0,\pi)$. If $D_\alpha> \alpha$, then  $D_\alpha$ is attained {by an even an decreasing function}, i.e. there exists $u_\alpha\in H$ {even and decreasing} s.t. $\|u_\alpha\|_H =1$ and $E_\alpha(u_\alpha) =D_\alpha$.
\end{lemma}
\begin{proof}
Let $(u_k)\subset H$ be a maximizing sequence for $E_\alpha$. W.l.o.g. we can assume $u_k\to u_\infty\in H$ weakly in $H$ and a.e. on $\R$. Moreover, up to replacing $u_k$ with its symmetric decreasing rearrangement, we can assume that $u_k$ is even and decreasing (see \cite{Park}). Since $\alpha\in (0,\pi)$ the sequence $e^{\alpha u_k^2} -1$ is bounded in $L^\frac{\pi}{\alpha}(\R)$, with $\frac{\pi}{\alpha}>1$. Then, by Vitali's theorem, we get $e^{\alpha u_k^2} \to e^{\alpha u_{\infty}^2} $ in $L^1_{\loc}(\R)$, and Lemma \ref{LemmaVanishing} yields 
\begin{equation}\label{EVan}
E_\alpha(u_k) = E_\alpha(u_\infty)+\alpha \bra{ \|u_k\|_{L^2(\R)}^2 - \|u_\infty\|_{L^2(\R)}^2}+o(1).
\end{equation}
This implies that $u_\infty\not\equiv 0$, since otherwise we have
$
E_\alpha(u_k)  = \alpha \|u_k\|_{L^2(\R)}^2 +o(1)  \le \alpha+o(1),
$
which contradicts the assumption $D_\alpha > \alpha$. Let us denote
$$L:=\limsup_{k\to\infty}  \|u_k\|^2_{L^2(\R)},\quad   \tau :=\frac{\|u_\infty\|^2_{L^2(\R)}}{L}.$$
Observe that $L,\tau\in (0,1]$.
Let us consider the sequence $v_k(x)= u_k(\tau x)$. Clearly, we have $v_k \rw v_\infty$ weakly in $H$, where $v_\infty(x):=u_\infty(\tau x)$. Moreover, since 
$$
\|u_\infty\|_{L^2}^2 =L, \quad \text{ and }\quad \|(-\Delta)^\frac{1}{4} v_\infty \|_{L^2}^2 \le \liminf_{k\to\infty }\|(-\Delta)^\frac{1}{4} v_k \|_{L^2}^2 = \liminf_{k\to\infty }\|(-\Delta)^\frac{1}{4} u_k \|_{L^2}^2 \le 1-L,
$$
we get $\|v_{\infty}\|_{H}\le 1$. By \eqref{EVan} we have
\begin{equation}\label{portion3}
D_\alpha \le E_\alpha(u_\infty) + \alpha L(1-\tau)   = \tau E_\alpha(v_\infty) + \alpha L(1-\tau)   \le  \tau D_\alpha+ \alpha L(1-\tau).
\end{equation}
If $\tau<1$, this implies $D_\alpha \le \alpha L \le \alpha$, contradicting the assumptions. 
Hence $\tau =1$ and \eqref{portion3} gives $D_\alpha = E_\alpha(u_\infty)$. Finally we have $\|u_0\|_{H}=1$, otherwise $E_\alpha(\frac{u_\infty}{\|u_\infty\|_H})> E_\alpha(u_\infty)=D_\alpha$.
\end{proof}

\begin{lemma}\label{sub}
There exists $\alpha^*\in (0,\pi)$ such that $D_\alpha>\alpha$ for any $\alpha\in (\alpha^*,\pi]$. In particular $D_\alpha$ is attained {by an even and decreasing function $u_\alpha$} for any $\alpha \in (\alpha^*,\pi)$ by Lemma \ref{novan}. 
\end{lemma}
\begin{proof}
This follows from Proposition \ref{Cg2bis} by continuity. Indeed  Proposition \ref{Cg2bis} gives $D_\pi > 2\pi e^{-\gamma} > \pi$. 
\end{proof}

%

\subsection{The critical case}
Next, we take a sequence $\alpha_k$ such that $\alpha_k\nearrow \pi$ as $k\to \infty$. For any  large $k$, Lemma \ref{sub} yields the existence  $u_k\in H$ {even and decreasing} such that $D_{\alpha_k}= E_{\alpha_k}(u_k)$. Each $u_k$ satisfies
$$
(-\Delta)^\frac{1}{2}u_k +u_k = \lambda_k u_k e^{\alpha_k u_k^2} ,
$$
and $\|u_k\|_H=1$. Note that $u_k\in C^\infty(\R)$ by elliptic estimates. Multiplying the equation by $u_k$ and using the basic inequality $t e^{t} \ge e^t-1$, for $t\ge 0$, we infer  
\[
\begin{split}
\frac{1}{\lambda_k} = \int_{\R} u_k^2 e^{\alpha_k u_k^2} dx    \ge \frac{1}{\alpha_k} E_{\alpha_k}(u_k) = \frac{1}{\alpha_k} D_{\alpha_k}.
\end{split}
\]
Since $D_{\alpha_k}\to D_{\pi}>0$, we get that $\lambda_k$ is uniformly bounded.

Then the sequence $u_k$ satisfies the alternative of Theorem \ref{compact}. If case (i) holds, then we can argue as in Lemma \ref{novan} and Lemma \ref{sub} and prove that $D_\pi$ is attained. Therefore, we shall assume by contradiction that case (ii) occurs.

Let $r_k$ and $\eta_k$ be as in Theorem \ref{compact}. Let $\tilde \eta_k$ denote the Poisson integral of $\eta_k$. 

\begin{prop}\label{convetaR}
We have $\tilde \eta_k\to \tilde \eta_\infty$ in $C^\ell_{\loc}(\overline{\R^2_+})$ for every $\ell\ge 0$, where
$$
\tilde \eta_\infty(x,y)=-\log\bra{(1+y)^2+x^2}
$$
is the Poisson integral (compare to \eqref{Poisson2}) of $\eta_{\infty}:=-\log\bra{1+x^2}$. 
\end{prop}
\begin{proof}
By Theorem \ref{compact} we know that $\eta_k \to \eta_\infty$ in $C^\ell_{\loc}(\R)$ and that $\eta_k$ is bounded in $L_{\frac{1}{2}}$. Then, we can repeat the argument of the proof of Proposition \ref{conveta}.
\end{proof}

\begin{rmk}\label{integralsR} As in \eqref{integrals}, the convergence $\eta_k\to \eta_\infty$ in $L^\infty_{\loc}(\R)$ implies 
$$
\lim_{k\to \infty} \int_{-r_kR}^{r_kR} \lambda_k \mu_k^{i} u_k^{2-i} e^{\alpha_k u_k^2} = \frac{1}{\pi}\int_{-\pi}^\pi e^{\eta_\infty} dx,
$$
for $i=0,1,2$ and for any $R>0$.
\end{rmk}

\begin{lemma}\label{conv 0 in L2R}
We have $u_k \to 0$ in $L^2(\R)$. 
\end{lemma}
\begin{proof}
Indeed, otherwise up to a subsequence we would have $\|(-\Delta)^\frac{1}{4} u_k\|_{L^2(\R)} \le \frac{1}{A}$ for some $A>1$. Consider, the function $v_k = (u_k-u_k(1))^+$. Then, $v_k\in \tilde H^{\frac{1}{2},2}(I)$ and $\|(-\Delta)^\frac{1}{4} v_k\|_{L^2(\R)}\le \frac{1}{A}$. The Moser-Trudinger inequality \eqref{stimaMT}  gives that $e^{\alpha_kv_k^2}$ is bounded in $L^A(\R)$. Since 
$$
u_k^2 \le (1+\eps)v_k^2 + \frac{1}{\eps} (u_k-v_k)^2
$$ 
and $|v_k-u_k|\le u_k(1)\to 0$ as $k\to \infty$, we get that $e^{\alpha_k u_k^2}$ is uniformly bounded in $L^p(\R)$ for every $1<p<A$. Therefore, we have 
$$
\int_{(-1,1)} (e^{\alpha_k u_k^2}-1) dx \to 0
$$
as $k\to \infty$. But then, by Lemma \ref{LemmaVanishing} we find  $D^\pi_g  \le \pi$, which contradicts Lemma \ref{sub}.
\end{proof}

\begin{lemma}\label{lemma1A su R}
For $A>1$, set  $u_k^A:=\min\left\{u_k,\frac{\mu_k}{A}\right\}$. Then we have
\begin{equation}
\limsup_{k\to \infty}\|\ql u_k^A\|_{L^2(\R)}^2 \le \frac{1}{A}.
\end{equation}
\end{lemma}

\begin{proof}
The proof is similar to the one of Lemma \ref{lemma1A}. We set $\bar u_k^A:=\min\left\{\tilde u_k,\frac{\mu_k}{A}\right\}$. Since $\bar u_{k}^A$ is an extension of $u_k^A$, using integration by parts and the harmonicity of $\tilde u_k$ we get
\begin{equation}\label{toprove2R}
\begin{split}
\|\ql u_k^A\|_{L^2(\R)}^2  \le \int_{\R^2_+}|\nabla \bar u_k^A|^2dx dy=\int_{\R^2_+}\nabla \bar u_k^A \cdot \nabla \tilde u_k dxdy &= - \int_{\R} u_k^A(x) \frac{\de \tilde u(x,0)}{\de y}dx\\
&=\int_{\R{}}(-\Delta)^{\frac 12} u_k u_k^Adx.
\end{split}
\end{equation}

Proposition \ref{convetaR} implies that $u_k^A(r_kx)=\frac{\mu_k}{A}$ for $|x|\le R$ and $k\ge k_0(R)$. Noting that $u_k^A\le u_k$ and using Lemma \ref{conv 0 in L2R}, and Remark \ref{integralsR}, we get
\[
\begin{split}
\int_{\R{}}(-\Delta)^{\frac 12} u_k u_k^A\, dx&\geq \int_{-Rr_k}^{Rr_k} \lambda_k u_k e^{\alpha_k u_k^2}u_k^A dx - \int_{\R} u_k u_k^A dx \\&
=\frac{1}{A}\int_{-Rr_k}^{Rr_k} \lambda_k \mu_k u_k e^{\alpha_k u_k^2}u_k^A dx + O(\|u_k\|_{L^2(\R)}^2)
\\&
\stackrel{k\to\infty}{\to} \frac{1}{\pi A}\int_{-R}^R e^{\eta_\infty} d\xi
\\&\stackrel{R\to\infty}{\to} \frac{1}{A}.
\end{split}
\]

Set now $v_k^A:=\left(u_k-\frac{\mu_k}{A}\right)^+$. With  similar computations we get
\[
\begin{split}
\int_{\R{}}(-\Delta)^{\frac 12} u_k v_k^A dx&\geq \int_{-Rr_k}^{Rr_k} \lambda_k u_k  v_k^A e^{\alpha_k u_k^2} dx + O(\|u_k\|_{L^2(\R)}^2)
\\ &\stackrel{k\to\infty}{\to}\frac{1}{\pi}\bra{1-\frac{1}{A}} \int_{-R}^R e^{\eta_\infty}d\xi\\
& \stackrel{R\to\infty}{\to} \frac{A-1}{A}.
\end{split}
\]
Since
$$
\int_{\R{}}(-\Delta)^{\frac 12} u_k u_k^A\, dx+ \int_{\R{}}(-\Delta)^{\frac 12} u_k v_k^A dx=\int_{\R}(-\Delta)^{\frac 12} u_k u_k dx = 1  - \|u_k\|_{L^2(\R)}^2 = 1+o(1),$$
we get that
\[
\lim_{n\to \infty} \int_{\R{}}(-\Delta)^{\frac 12} u_k u_k^A\, dx = \frac{1}{A}.
\] 
Then, we conclude using \eqref{toprove2R}.
\end{proof}

\begin{prop}\label{stimaalto su R} We have
\begin{equation}\label{Cglim su R}
D_\pi= \lim_{k\to\infty}\frac{1}{\lambda_k\mu_k^2}.
\end{equation}
Moreover
\begin{equation}\label{muklambdak su R}
\lim_{k\to\infty} \mu_k \lambda_k =0.
\end{equation}
\end{prop}

\begin{proof} Fix $A>1$ and write
\[\begin{split}
\int_{\R}\bra{e^{\alpha_k u_k^2}-1}dx&= \int_{\{u_k\le\frac{\mu_k}{A}\}\cap (-1,1)}\bra{e^{\alpha_k u_k^2}-1}dx   + \int_{\{u_k\le\frac{\mu_k}{A}\}\cap (-1,1)^c} \bra{e^{\alpha_k u_k^2}-1}dx \\ &\quad +\int_{\{u_k>\frac{\mu_k}{A}\}} \bra{e^{\alpha_k u_k^2}-1}dx\\
&=:(I)+(II)+ (III). 
\end{split}\]
Using Lemmas \ref{L2conv} and \ref{lemma1A su R} together with Theorem \ref{MT3} we see that
$$(I)\le \int_{-1}^{1}\bra{e^{\alpha_k (u_k^A)^2}-1}\to 0\quad \text{as }k\to\infty$$
since $e^{\alpha_k (u_k^A)^2}-1$ is uniformly bounded in $L^p$, for any $1\le p<A$. By \eqref{decay} and Lemma \ref{conv 0 in L2R}, we find 
$$
(II)\le  \int_{(-1,1)^c} \bra{e^{\alpha_k u_k^2}-1} dx \le C \int_{\R} u_k^2 dx  \to 0\quad \text{ as } k\to \infty.
$$
We now estimate\[\begin{split}
(III)&\le \frac{A^2}{\lambda_k \mu_k^2} \int_{\{u_k>\frac{\mu_k}{A}\}}\lambda_k u_k^2 e^{\alpha_k u_k^2}dx \le \frac{A^2}{\lambda_k \mu_k^2} (1+o(1)),
\end{split}\]
with $o(1)\to 0$ as $k\to\infty$, where we used that
\[\begin{split}
\int_{I\cap\{u_k>\frac{\mu_k}{A}\}} \lambda_k u_k^2 e^{\alpha_k u_k^2} dx &\le \|u_k\|_H^2=1.
\end{split}\]

Letting $A\downarrow 1$, this gives
$$\sup_{H} E_\pi \le \lim_{k\to\infty}\frac{1}{\lambda_k\mu_k^2}.$$
The converse inequality follows from Remark \ref{integralsR}:
\[\begin{split}
\int_{\R}\bra{e^{\alpha_k u_k^2}-1}dx&\ge \int_{-Rr_k}^{Rr_k}  e^{\alpha_k u_k^2}dx +o(1)=\frac{1}{\lambda_k\mu_k^2} \left( \int_{-R}^R e^{\eta_\infty}dx +o(1) \right)+o(1).
\end{split}\]
with $o(1)\to 0$ as $k\to\infty$. Letting $R\to\infty$  we obtain \eqref{Cglim su R}.

Finally, \eqref{muklambdak su R} follows at once from \eqref{Cglim su R}, because otherwise we would have $D_\pi=0$, which is clearly impossible.
\end{proof}

\begin{lemma}\label{convdelta} We have
$$f_k:=\lambda_k \mu_k u_k e^{\alpha_k u_k^2}\rightharpoonup \delta_0$$
as $k\to \infty$, in the sense of Radon measures in $\R$.
\end{lemma}
\begin{proof}
The proof follows step by step the one Proposition  \ref{convdelta0}, with \eqref{stimaMTR}, Proposition \ref{convetaR}, Remark \ref{integralsR} Lemma \ref{conv 0 in L2R} and Lemma \ref{lemma1A su R} used in place of  \eqref{stimaMT}, Proposition \ref{conveta},  \eqref{integrals}, Lemma \ref{uinfty0} and Lemma \ref{lemma1A}. We omit the details.
\end{proof}

%

For $x\in \R$, let $G_x$ be the Green function of $\hl+ Id$ on $\R$ with singularity at $x$. In the following we denote $G:= G_0$.   By translation invariance, we get $G_x(y)= G(y-x)$ for any $x,y\in \R$, $x\neq y$. Moreover, the inversion formula for the Fourier-transform implies that
\begin{equation}\label{formula GreenR}
G(x)= \frac{1}{2} \sin |x| - \frac{1}{\pi} \sin( |x|) \mathrm{Si}(|x|) - \frac{1}{\pi} \cos (|x|)  \mathrm{Ci} (|x|), 
\end{equation}
where 
$$
\mathrm{Si}(x)= \int_{0}^x \frac{\sin t}{t} dt \qquad \text{ and }\qquad \mathrm{Ci}(x)= -\int_x^{+\infty} \frac{\cos t}{t} dt.
$$
We recall that the identity
\begin{equation}\label{CiIdentity}
\mathrm{Ci}(x) = \log x +\gamma +\int_{0}^x \frac{\cos t -1}{t} dt
\end{equation}
holds for any $x\in \R\setminus \{0\}$, where $\gamma$ denotes the Euler-Mascheroni constant see e.g. \cite[Chapter 12.2]{Gamma}.

\begin{prop}\label{PropGR} The function $G$ satisfies the following properties.
\begin{enumerate} 
\item We have $G\in C^\infty(\R\setminus \{0\})$ and 
\begin{equation}\label{Asym G}
G(x)=-\frac{1}{\pi} \log |x| -\frac{\gamma}{\pi} + O(|x|), \quad G'(x)= -\frac{1}{\pi x} +O(1), \quad {\text{as }x\to 0.}
\end{equation}
\item We have $G(x)= O(|x|^{-2})$ and $G'(x)=O(|x|^{-3})$ as $|x|\to \infty$.
\item Let $\tilde G$ be the Poisson extension of $G$. There exists a function $f\in C^{1}(\ov{ \R^{2}_+})$ such that $f(0,0)=0$ and 
\begin{equation}\label{ExpGtildeR}
\tilde G(x,y)=-\frac{1}{\pi} \ln{|(x,y)|}-\frac{\gamma}{\pi} + \frac{x}{\pi} \arctan \frac{x}{y} - \frac{y}{2\pi} \log(x^2+y^2) + f(x,y) \quad \text{in } \R^2_+.
\end{equation}
\end{enumerate}
\end{prop}

\begin{proof}
Property \emph{1.} follows directly by formula \eqref{formula GreenR} and the identity in \eqref{CiIdentity}. Similarly, since 
\[
\begin{split}
Si(t) = \frac{\pi}{2}-\frac{\cos t}{t} -\frac{\sin t}{t^2}+ O(t^{-3}), \qquad  Ci(t)=\frac{\sin t}{t}-\frac{\cos t}{t^2}+ O(t^{-3}),
\end{split}
\]
as $t\to +\infty$, we get \emph{2}. 

Given $R>0$, let $\psi\in C^\infty_c(\R)$ be a cut-off function with $\psi \equiv 1$ on $(-R,R)$. Let us denote
$g_0:= -\frac{1}{\pi} \log |\cdot|-\frac{\gamma}{\pi}$, $g_1:= \frac{1}{2}|\cdot| \psi$, $g_2:= G-g_0-g_1$.  By Proposition \ref{Uniq}, we have 
$$\tilde g_0(x,y)= -\frac{1}{\pi}\log |(x,y)|-\frac{\gamma}{\pi},\quad (x,y)\in \R^2.$$ Denoting $\theta(x,y) := \arctan{\frac{x}{y}}$ the angle between the $y$-axis and the segment connecting the origin to $(x,y)$, the function
$$
h(x,y):=\tilde g_1(x,y) - \frac{1}{\pi} x  \,\theta(x,y) +\frac{1}{2\pi}y \log (x^2+y^2)
$$ is harmonic in $\R^2_+$, continuous  on $\ov \R^2_+$, and identically $0$ on $(-R,R)\times \R$. By \cite[Theorem C]{Rudin}, we get that $h\in C^\infty(\ov{\R^2_+} \cap B_{R}(0,0))$. Finally, note  that formula \eqref{formula GreenR} implies $g_2\in C^2(\R)$ and $g_2(0)=0$. Hence, standard elliptic regularity yields $\tilde g_2 \in C^{1,\alpha}(\ov{\R^2_+} \cap B_R(0,0))$, for any $\alpha \in (0,1)$. In particular $\tilde g_2(0,0)=g_2(0)=0$. 
\end{proof}

%

\begin{lemma}\label{L2conv} We have  $\mu_k u_k \to G$ in  ${L^2(\R)} \cap L^\infty (\R \setminus (-\eps,\eps))$, for any $\eps >0$,
\end{lemma}
\begin{proof}
Let us set $v_k:= \mu_k u_k - G$ and $f_k = \mu_k \lambda_k u_k e^{\alpha_k u_k^2}$. By Lemma \ref{convdelta} we have $\|f_k\|_{L^1(I)}\to 1$ as $k\to +\infty$, $I =(-1,1)$.  Then, arguing as in Lemma \ref{convgreen}, we get
\begin{equation}\label{Green1R}
\begin{split}
|v_k(x)| & =  \left| \int_{\R} G(y-x) f_k(y)\,dy - G(x) \right| \\
 & \le \int_{I}  |G(x-y)-G(x)| f_k(y)  \,dy   + \underbrace{|\|f_k\|_{L^1(I)}-1|}_{=o(1)}\, |G(x)| + \underbrace{\int_{\R\setminus I} G(x-y) f_k(y) dy.}_{=: w_k(x)}
\end{split}
\end{equation}
Using \eqref{decay}, Lemma \ref{conv 0 in L2R} and \eqref{muklambdak su R}, we get that $f_k \to 0$ in $L^2(\R\setminus I)$.  In particular 
$$
\left|w_k(x)\right|\le \|f_k\|_{L^2(\R\setminus I)} \|G\|_{L^2(\R)} \to 0. 
$$
Fix $\sigma\in (0,1)$ and assume $|x|\ge \sigma$. If we further take  $|y|\le \frac{\sigma}{2}$, then Proposition \ref{PropGR} implies
\[\begin{split}
|G(x-y)-G(x)|& \le C |y|,
\end{split}\]
where $C$ is a constant depending only on $\sigma$.  Thus, for any $\eps \in (0,\frac{\sigma}{2})$,  we can write 
\begin{equation}\label{Green3R}\begin{split}
|v_k(x)|& \le \int_{I}  |G(x-y)-G(x)| f_k(y)  \,dy + o(1) \|G\|_{L^\infty(\R\setminus (-\sigma,\sigma))}+o(1)\\
&  \le  C \int_{-\eps}^\eps  |y| f_k(y)  \,dy + \int_{I\setminus(-\eps,\eps)}  |G(x-y)| f_k(y)  \,dy +  |G(x)| \int_{I\setminus(-\eps,\eps)} f_k(y)  \,dy +o(1)\\
& \le C \eps \|f_k\|_{L^1(I)} + \|f_k\|_{L^\infty(I \setminus (-\eps,\eps))}  \left( \|G\|_{L^1(\R)} +  \|G\|_{L^\infty(\R\setminus (-\sigma,\sigma))} \right) +o(1) \\
& \le  C \eps +o(1),
\end{split}\end{equation}
where $o(1)\to 0$ as $k\to \infty$ (depending on $\eps$ and $\sigma$). 
Here, we used that $f_k\to 0$ in $L^\infty(\R\setminus (-\eps,\eps))$  by \eqref{decay} and \eqref{muklambdak su R}. Since $\eps$  is arbitrarily small, \eqref{Green3R} shows that $v_k\to 0$ in $L^\infty(\R \setminus (-\sigma,\sigma))$. 

Next, we prove the $L^2$ convergence. First, H\"older's inequality and Fubini's theorem give
$$
\|w_k\|_{L^2(\R)}^2 = \int_{\R} \left( \int_{\R\setminus I} G(x-y) f_k(y) dy \right)^2 dx \le  \|G\|_{L^1(\R)}^2 \|f_k\|_{L^2(\R\setminus I)}^2\to 0 
$$
as $k\to \infty$. With a similar argument, after integrating \eqref{Green1R} and using the triangular inequality in $L^2$, we find
\[\begin{split}
\|v_k\|_{L^2(\R)} &\le \left( \int_\R \left( \int_I |G(x-y)-G(x)|f_k(y ) \, dy\,\right)^2 dx\right)^\frac{1}{2}  + | \|f_k\|_{L^1(I)} -1| \|G\|_{L^2(\R)} + \|w_k\|_{L^2(\R)}\\
& \le \underbrace{\bra{\int_{I} f_k(y)dy}^\frac{1}{2}}_{=1+o(1)} \left( \int_I  f_k(y)  \int_{\R}  |G(x-y)-G(x)|^2 \,dx\, dy \right)^\frac{1}{2} + o(1).
\end{split}
\]
Since $G\in L^2(\R)$, the function $\psi(y):= \int_{\R} |G(x-y)-G(x)|^2 dx$ is continuous on $\R$ and $\psi(0)=0$. Let $\ph\in C(\R)$ be a compactly supported function such that $\ph\equiv \psi$ on $I$. Then, Lemma \ref{convdelta} implies
\[
\int_I  f_k(y)  \int_{\R}  |G(x-y)-G(x)|^2 \,dx\, dy = \int_{I} f_k(y)\ph(y) dy= \int_{\R} f_k(y)\ph(y) dy + o(1)=o(1),
\]
as $k\to \infty$, and the conclusion follows.

\end{proof}

Repeating the argument of Proposition \blu{\ref{convgreen}}, we get the following:

\begin{lemma}\label{ConvGreenR}
We have $\mu_k \tilde u_k \rightarrow \tilde G$  in $C^0_{\loc}(\overline{\R^2_+}\setminus\{(0,0)\}) \cap C^1_{\loc}( \R^2_+)$, where $\tilde{G}$ is the Poisson extension of $G$. 
\end{lemma}

With Proposition \ref{convetaR} and Lemma \ref{ConvGreenR} we can give an upper bound on $D_\pi$.

\begin{prop} Under the assumption that $\mu_k\to \infty$ as $k\to\infty$, we have
$D_\pi \le 2\pi e^{-\gamma}$. 
\end{prop}
\begin{proof} For a fixed and small $\delta>0$ set
$$a_k:=\inf_{B_{Lr_k}\cap \R^2_+} \tilde u_k,\quad  b_k:=\sup_{B_\delta\cap \R^2_+} \tilde u_k,\quad \tilde v_k:= (\tilde u_k \wedge a_k)\vee b_k.$$
Recalling that $\|\nabla \tilde u_k\|_{L^2(\R)}^2=\|(-\Delta)^\frac{1}{4} u_k\|^2_{L^2(\R)} =1-\|u_k\|^2_{L^2(\R)}$, we have
\begin{equation}
\int_{(B_\delta \setminus B_{Lr_k})\cap \R^2_+}|\nabla \tilde v_k|^2dxdy\le 1-\|u_k\|_{L^2}^2 -\int_{\R^2_+\setminus B_\delta} |\nabla \tilde u_k|^2dx -\int_{\R^2_+\cap B_{Lr_k}} |\nabla \tilde u_k|^2dx
\end{equation}
Clearly the left-hand side bounds
\[\begin{split}
\inf_{\substack{\tilde u|_{\R^2_+\cap \de B_{Lr_k}}=a_k\\ \tilde u|_{\R^2_+\cap \de B_\delta}=b_k}} \int_{(B_\delta \setminus B_{Lr_k})\cap \R^2_+}|\nabla \tilde u|^2dxdy=\pi\frac{(a_k-b_k)^2}{\log \delta -\log (Lr_k)}.
\end{split}\]
Using Proposition \ref{convetaR}, Proposition \ref{PropGR} and Lemma \ref{ConvGreenR} we obtain
\begin{equation}\label{a and b}
a_k=\mu_k+\frac{-\frac{1}{\pi}\log L+O(L^{-1})+o(1)}{\mu_k} \quad \text{ and } \quad b_k=\frac{-\frac{1}{\pi}\log \delta -\frac{\gamma}{\pi} +O(\delta |\log \delta|)+o(1)}{\mu_k},
\end{equation}
where $o(1)\to0$ as $k\to\infty$ for fixed $L>0$, $\delta>0$,  and $|O(L^{-1})|\le C L^{-1}$, $|O(\delta |\log \delta|)|\le C\delta |\log \delta|$,  uniformly for $\delta$ small, and $L$, $k$ large. Still with Proposition \ref{convetaR} we get
\[
\begin{split}
\lim_{k\to\infty}\mu_k^2\int_{B_{Lr_k}^+} |\nabla \tilde u_k|^2 dxdy &=\frac{1}{4\pi^2}\int_{B_L^+}|\nabla \tilde \eta_\infty|^2dxdy\\
&=\frac{1}{\pi}\log\frac{L}{2} +O\bra{\frac{\log L}{L}}.
\end{split}\]
Similarly Lemma \ref{ConvGreenR} and Proposition \ref{PropGR} yield
\[
\begin{split}
\liminf_{k\to\infty} \mu_k^2\int_{\R^2_+\setminus B_\delta}|\nabla \tilde u_k|^2dxdy
&\ge \int_{\R^2_+\setminus B_\delta}|\nabla \tilde G|^2dxdy\\
&=\int_{\R^2_+\cap \partial B_\delta} -\frac{\de \tilde G}{\de r} \tilde G d\sigma + \int_{(\R \times \{0\} )\setminus B_\delta} -\frac{\de \tilde G(x,0)}{\de y}  G(x)dx \\
&=\int_{\R^2_+\cap \partial B_{\delta}} \bra{\frac{1}{\pi \delta}+O(|\log \delta|)}\bra{-\frac{1}{\pi}\log\delta -\frac{\gamma}{\pi}+O(\delta|\log \delta|)} d\sigma  \\
& \qquad -  \int_{\R \setminus (-\delta,\delta)} G(x)^2dx   \\
&= -\frac{1}{\pi}\log \delta -\frac{\gamma}{\pi} -\|G\|_{L^2(\R)}^2 + O(\delta \log^2\delta),
\end{split}
\]
where we used that 
$$-\frac{\de \tilde G(x,0)}{\de y}=\hl G(x)= -G(x),\quad \text{for }x\in \R\setminus \{0\}.$$ 
From Lemma \ref{L2conv} we get that  $\mu_k u_k\to G$ in $L^2(\R)$, hence 
$$
\|u_k\|_{L^2(\R)}^2 = \frac{\|G\|_{L^2(\R)}^2+o(1)}{\mu_k^{2}}
$$
as $k\to +\infty$.
We then get
\[\begin{split}
\frac{\pi(a_k-b_k)^2}{\log\delta -\log(Lr_k)}&\le 1-\frac{-\frac{1}{\pi}\log \delta-\frac{\gamma}{\pi} +O(\delta\log^2 \delta)+\frac{1}{\pi}\log\frac{L}{2}+O\bra{\frac{\log L}{L}}+o(1)}{\mu_k^2}.
\end{split}\]
Using  \eqref{a and b}  and rearranging as in the proof of Proposition \ref{Cg1}, we find
\[\begin{split}
\log\frac{1}{\lambda_k \mu_k^2}&\le  \bra{1-\frac{\alpha_k}{\pi}}(\log L-\log\delta)+(\alpha_k-\pi)\mu_k^2 +(\frac{\alpha_k}{\pi} - 2)\gamma
+\frac{\alpha_k}{\pi}\log 2 +\log\alpha_k\\
&\quad +O(\delta \log^2 \delta )+ O\bra{\frac{\log L}{L}}+o(1),
\end{split}\]
with $o(1)\to 0$ as $k\to \infty$. Then, recalling that $\alpha_k\uparrow \pi$, letting $k\to \infty$ first and then $L\to \infty$, $\delta\to 0$, we obtain
$$\limsup_{k\to\infty}\log  \frac{1}{\lambda_k \mu_k^2} \le -\gamma+ \log(2\pi),$$
and using Proposition \ref{stimaalto su R} we conclude.
\end{proof}


\begin{prop}\label{Cg2bis}
There exists a function $u\in H^{\frac{1}{2},2}(\R)$ such that $\|u\|_{H}\le 1$ and $E_\pi(u)>2\pi e^{-\gamma}.$
\end{prop}

\begin{proof} For $\ve>0$ choose $L=L(\ve)>0$ such that as $\ve\to 0$ we have $L\to\infty$ and $L\ve\to 0$. Fix
$$\Gamma_{L\ve}:=\left\{(x,y)\in \R^2_+: \tilde G(x,y)=\gamma_{L\ve}:=\min_{\R^2_+\cap\de B_{L\ve}} \tilde G\right\}, $$
and
$$\Omega_{L\ve}:=\left\{ (x,y)\in \R^2_+: \tilde G(x,y)>\gamma_{L\ve} \right\}.$$
By the maximum principle we have $\R^2_+\cap B_{L\ve}\subset \Omega_{L\ve}$. Notice also that Proposition \ref{PropGR} gives
$$\gamma_{L\ve} =-\frac{1}{\pi}\log(L\ve)-\frac{\gamma}{\pi}+O(L\ve |\log (L \eps)|).$$
and $\Omega_{L\ve}\subseteq \R^2_+\cap B_{2L_\eps}$. For suitable constants  $B,c\in \R$ to be fixed we set
\[
U_\ve(x,y):=\left\{
\begin{array}{ll}
\displaystyle c- \frac{\log\bra{\frac{x^2}{\ve^2} +\bra{1+\frac{y}{\ve}}^2 }+ 2B }{2\pi c}& \text{for } (x,y)\in B_{L\ve}(0,-\ve)\cap \R^2_+\\
\displaystyle\rule{0cm}{0.5cm} \frac{\gamma_{L\ve}}{c} &\text{for }(x,y)\in \Omega_{L\ve}\setminus  B_{L\ve}(0,-\ve)\\
\displaystyle\rule{0cm}{0.8cm}  \frac{\tilde G(x,y)}{c} &\text{for }(x,y)\in \R^2_+\setminus \Omega_{L\ve}.
\end{array}
\right.
\]
Observe that $\R^2_+\cap B_{L\eps}(0,-\eps)\subseteq \R^2\cap B_{L\eps} \subseteq \Omega_{L\eps}$. We choose $B$ in order to have continuity on $\R^2_+\cap \de B_{L\ve}(0,-\ve)$, i.e. we impose
$$\frac{-\log L^2 -2B }{2\pi c}+ c=\frac{\gamma_{L\ve}}{c},$$
which gives the relation
\begin{equation}\label{contoB R}
B=\pi c^2+\log\ve +\gamma +O(L\ve |\log (L\eps)|).
\end{equation} 
This choice of $B$ also implies that the function $c U_\eps$ does not depend on the value of $c$. Then we can choose $c$ by imposing  \begin{equation}\label{norma1}
\|\nabla U_\ve\|_{L^2(\R^2_+)}^2+\|u_\ve\|_{L^2(\R)}^2=1,
\end{equation}
where we set $u_\eps(x) = U_\eps(x,0)$. Since the harmonic extension $\tilde u_\ve$ minimizes the Dirichlet energy among extensions with finite energy, we have
\[\begin{split}
\|\ql u_\ve \|_{L^2(\R)}^2&=\int_{\R^2_+}|\nabla \tilde u_\ve|^2dxdy
\le \int_{\R^2_+}|\nabla U_\ve|^2dxdy,
\end{split}\]
and \eqref{norma1} implies $\|u_\ve\|_{H^{\frac12,2}(\R)}^2\le 1$.  

In order to obtain a more precise expansion of $B$ and $c$ we compute
\begin{equation}\label{interior}\begin{split}
\int_{B_{L\ve}(0,-\ve)\cap \R^2_+} |\nabla U_\ve|^2dxdy&=\frac{1}{4\pi^2 c^2}\int_{B_L(0,-1)\cap \R^2_+}|\nabla \log(x^2+(1+y)^2)|^2dxdy\\
&=\frac{\frac{1}{\pi}\log\bra{\frac{L}{2}}+O\bra{\frac{\log L}{L}}}{c^2},
\end{split}\end{equation}
and
\[\begin{split}\label{split}
\int_{\R^2_+\setminus\Omega_{L\ve}}|\nabla U_\ve|^2 dxdy&=\frac{1}{c^2}\int_{\R^2_+\setminus\Omega_{L\ve}}|\nabla\tilde G|^2 dxdy\\
&=-\frac{1}{c^2}\int_{\R^2_+\cap \de \Omega_{L\ve}}\frac{\de \tilde G}{\de \nu} \tilde G d\sigma-\frac{1}{c^2}\int_{(\R\times\{0\})\setminus \bar \Omega_{L\ve}} \frac{\de \tilde G}{\de y} \tilde G d\sigma\\
&=(I)+(II).
\end{split}\]
By the divergence theorem we have for $\tau < L\ve$ and letting $\tau \to 0$,
\begin{equation}\label{stima(I)}\begin{split}
(I)&=- \frac{\gamma_{L\ve}}{c^2}\int_{(\R\times \{0\})\cap (\ov \Omega_{L\ve} \setminus B_\tau) } \frac{\de \tilde G}{\de \nu}  d\sigma -\frac{\gamma_{L\ve}}{c^2}\int_{\R^2_+\cap \de B_\tau } \frac{\de \tilde G}{\de \nu}  d\sigma \\
&=\frac{\gamma_{L\ve}}{c^2} \left( \int_{(\R\times \{0\})\cap \ov \Omega_{L\ve}} G  d\sigma    +1 \right)\\
&=\frac{\gamma_{L\ve}}{c^2}(1+O(L\ve \log (L\ve)) )\\
&=\frac{\frac{1}{\pi}\log\bra{\frac{1}{L\ve}} - \frac{\gamma}{\pi} +O(L\ve \log^2(L\ve))}{c^2},
\end{split}\end{equation}
where in the third identity we used that $\Omega_{L\ve}\subset B_{2L\ve}$ for $L\ve$ small enough. Observe also that
$$\|u_\ve\|_{L^2(\R)}^2= \frac{1}{c^2}\int_{(\R\times\{0\})\setminus \bar \Omega_{L\ve}} G^2 dx +\frac{O(L\ve \log^2(L\ve)) }{c^2} = - (II) + \frac{O(L\ve \log^2(L\ve)) }{c^2} .$$
Together with \eqref{norma1}-\eqref{stima(I)} this gives 
$$-\log \ve -\log 2 -\gamma+O(L\ve \log^2(L\ve))+ O\bra{\frac{\log L}{L}}=\pi c^2,$$
which, together with \eqref{contoB R}, implies
$$B=-\log 2+ O(L\ve \log^2(L\ve))+ O\bra{\frac{\log L}{L}}.$$
Now, observe that $B_{L\eps}(0,-\eps) \cap  (\R \times \{0 \}) = (-\eps \sqrt{L^2-1},\eps \sqrt{L^2-1})$ and that 
\[\begin{split}
\int_{-\ve\sqrt{L^2-1}}^{\ve\sqrt{L^2-1}}e^{\pi u_\ve^2}dx 
&=\ve  \int_{-\sqrt{L^2-1}}^{\sqrt{L^2-1}} \exp\bra{\pi \bra{c-\frac{\log(1+x^2)+2B}{2\pi c}}^2}dx\\
&>\ve e^{\pi c^2-2B} \int_{-\sqrt{L^2-1}}^{\sqrt{L^2-1}}\frac{1}{1+x^2}dx\\
&=2e^{-\gamma+ O(L\ve \log^2(L\ve))+ O\bra{\frac{\log L}{L}}}\pi\bra{1+O\bra{\frac 1L}}\\
&=2\pi e^{-\gamma}+O(L\ve \log^2(L\ve))+ O\bra{\frac{\log L}{L}}.
\end{split}\]
Moreover
\[\begin{split}
\int_{(\R\times \{0\})\setminus \bar \Omega_{L\ve}} \bra{e^{\pi u_\ve^2}-1}dx & \ge \int_{(\R\times \{0\})\setminus \bar \Omega_{L\ve}} \pi u_\ve^2dx
=\frac{1}{c^2} \int_{(\R\times \{0\})\setminus \bar \Omega_{L\ve}} \pi G^2dx=:\frac{\nu_{L\ve}}{c^2},
\end{split}\]
with
$$\nu_{L\ve}>\nu_{\frac{1}{2}}>0,\quad \text{for } L\ve<\frac{1}{2}.$$
Now choose $L=\log^2\ve$ to obtain
$$O(L\ve \log^2(L\ve))+ O\bra{\frac{\log L}{L}}= O\bra{\frac{\log\log\ve}{\log^2\ve}}= o\bra{\frac{1}{c^2}},$$
so that
\[\begin{split}
E_\pi (u_\eps) = \int_{\R} \bra{e^{\pi u_\ve^2}-1}dx &\ge 2\pi e^{-\gamma} +\frac{\nu_{\frac12}}{c^2}+o\bra{\frac{1}{c^2}}>2\pi e^{-\gamma}
\end{split}\]
for $\ve$ small enough.
\end{proof}

\noindent\emph{Proof of Theorem \ref{extR} (completed).}
By Propositions \ref{Cg2} and \ref{Cg2bis}, we know that $\mu_k\le C$. Then, by dominated convergence theorem we have $e^{\alpha_k u_k^2}\to e^{\pi u_\infty^2}$ in $L^1_{\loc}(\R)$.  
Then, by  Lemma \ref{novan}, we  infer
\begin{equation}\label{residual}
E_{\alpha_k}(u_k) = E_{\pi}(u_\infty)+ \pi (\|u_k\|_{L^2(\R)}^2-\|u_{\infty}\|_{L^2(\R)}^2)  + o(1). 
\end{equation}
This implies that $u_\infty\not\equiv 0$, otherwise we would have $E_{\alpha_k}(u_k)  \le  \pi \|u_k\|_{L^2(\R)}^2  +o(1) \le \pi +o(1)$, which contradicts the strict inequality $D_\pi > 2\pi e^{-\gamma}>\pi$, since $E_{\alpha_k}(u_k)=D_{\alpha_k}\to D_{\pi}$ as $k\to \infty$.

Let us denote $L:=\limsup_{k\to\infty} \|u_k\|^2_2$, $\tau =\frac{\|u_\infty\|^2_{L^2(\R)}}{L}$ and observe that  $L,\tau\in (0,1]$.  Let us consider the sequence $v_k(x)= u_k(\tau x)$. Clearly, we have $v_k \rw v_\infty$ in $H$, where $v_\infty(x):=u_\infty(\tau x)$. Since 
$$
\|v_\infty\|_{L^2}^2 =L, \quad \text{ and }\quad \|(-\Delta)^\frac{1}{4} v_\infty \|_{L^2}^2 \le \liminf_{k\to\infty }\|(-\Delta)^\frac{1}{4} v_k \|_{L^2}^2 = \liminf_{k\to\infty }\|(-\Delta)^\frac{1}{4} u_k \|_{L^2}^2 \le 1-L,
$$
we get $\|v_\infty\|_{H^{\frac{1}{2},2}}\le 1$. By \eqref{residual} we have
$$
D_\pi = E_\pi(u_\infty) + \pi L(1-\tau)= \tau E_\pi(v_\infty) + \pi L(1-\tau) \le  \tau D_\pi + \pi L(1-\tau).
$$
If $\tau<1$, this implies $D_\pi \le \pi L,$
which is not possible. Hence, we must have $\tau =1$ and $E_\pi(u_\infty)= D_\pi$.
\hfill$\square$

\appendix
\section{Appendix: The half-Laplacian on \texorpdfstring{$\R$}{R}}

For $u\in \mathcal{S}$ (the Schwarz space of rapidly decaying functions) we set
\begin{equation}\label{fraclapl0}
\widehat{(-\Delta)^s u}(\xi)=|\xi|^{2s}\hat u(\xi),\quad \hat{f}(\xi):=\int_{\R}f(x)e^{-ix\xi}dx.
\end{equation}
One can prove that it holds (see e.g.)
\begin{equation}\label{fraclapl}
(-\Delta)^s u(x)=K_s P.V.\int_{\R{}}\frac{u(x)-u(y)}{|x-y|^{1+2s}}dy:=K_s \lim_{\varepsilon\to 0}\int_{\R{}\setminus [-\ve,\ve]}\frac{u(x)-u(y)}{|x-y|^{1+2s}}dy,
\end{equation}
from which it follows that 
$$\sup_{x\in \R}|(1+x^{1+2s})(-\Delta)^s \varphi(x)|<\infty,\quad\text{for every }\varphi\in \mathcal{S}\,.$$
Then one can set
\begin{equation}\label{L12}
L_s(\R):=\left\{u\in L^1_{\loc}(\R):\|u\|_{L_s}:=\int_{\R}\frac{|u(x)|}{1+|x|^{1+2s}}dx<\infty   \right\},
\end{equation}
and for every $u\in L_{s}(\R)$ one defines the tempered distribution $(-\Delta)^su$ as
\begin{equation}\label{fraclapl2}
\langle (-\Delta)^s u,\varphi\rangle :=\int_{\R} u(-\Delta)^s \varphi dx =\int_{\R}u\,\mathcal{F}^{-1}(|\xi|\hat \varphi(\xi))\,dx,\quad\text{for every }\varphi \in\mathcal{S}.
\end{equation}
Moreover we will define for $p\ge 1$ and $s\in (0,1)$
\begin{equation}\label{defHsp}
H^{s,p}(\R):=\{u\in L^p(\R): (-\Delta)^\frac{s}{2} u \in L^p(\R)\}.
\end{equation}
In the case $s=\frac{1}{2}$ we have $K_\frac12=\frac{1}{\pi}$ in \eqref{fraclapl} and a simple alternative definition of $\hl$ can be given via the Poisson integral. For $u\in L_{\frac{1}{2}}(\R)$ define the Poisson integral
\begin{equation}\label{Poisson2}
\tilde u(x,y):=\frac{1}{\pi}\int_{\R}\frac{yu(\xi)}{(y^2+(x-\xi)^2)}d\xi, \quad y>0,
\end{equation}
which is harmonic in $\R^2_+=\R\times(0,\infty)$ and satisfies the boundary condition $\tilde u|_{\R\times \{0\}} =u$ in the following sense:

\begin{prop} If $u\in L^\frac{1}{2}(\R)$, then $\tilde u(\cdot,y)\in L^1_{\loc}(\R)$ for $y\in (0,\infty)$ and $\tilde u(\cdot,y)\to u$ in the sense of distributions as $y\to 0^+$.  If $u\in L^\frac{1}{2}(\R)\cap C((a,b))$ for some interval $(a,b)\subseteq \R$, then $\tilde u$ extends continuously to $(a,b)\times \{0\}$ and $\tilde u (x,0)=u(x)$ for any $x\in (a,b)$.   If $u\in H^\frac{1}{2}(\R)$, then $\tilde u\in H^1(\R^2_+)$, the identity $\|\nabla \tilde u \|_{L^2(\R^2_+)} = \|(-\Delta)^\frac{1}{4}u \|_{L^2(\R)}$ holds,  and $\tilde u|_{\R\times\{0\} }= u$ in the sense of traces. 
\end{prop}

Then we have (see e.g \cite{CaffSil})
\begin{equation}\label{fraclapl3}
\hl u =- \frac{\de \tilde u}{\partial y}\bigg|_{y=0},
\end{equation}
where the identity is pointwise if $u$ is regular enough (for instance $C^{1,\alpha}_{\loc}(\R)$), and has to be read in the sense of tempered distributions in general, with
\begin{equation}\label{fraclapl3b}
\bigg\langle -\frac{\de \tilde u}{\partial y}\bigg|_{y=0},\varphi\bigg\rangle:=\bigg\langle u, -\frac{\de \tilde \varphi}{\partial y}\bigg|_{y=0} \bigg\rangle,\quad \varphi\in\mathcal{S},\quad\tilde\varphi\text{ as in \eqref{Poisson2}}.
\end{equation}

More precisely:

\begin{prop}\label{lapeq} If $u\in L_{\frac{1}{2}}(\R)\cap C^{1,\alpha}_{\loc}((a,b))$ for some interval $(a,b)\subset\R$ and some $\alpha\in (0,1)$, then the tempered distribution $\hl u$ defined in \eqref{fraclapl2} coincides on the interval $(a,b)$ with the functions given by \eqref{fraclapl} and \eqref{fraclapl3}. For general $u\in L_\frac12(\R)$ the definitions \eqref{fraclapl2} and \eqref{fraclapl3} are equivalent, where the right-hand side of \eqref{fraclapl3} is defined by \eqref{fraclapl3b}.
\end{prop}

It is known that the Poisson integral of a function $u\in L^\frac{1}{2}(\R)$ is the unique harmonic extension of $u$ under some growth constraints at infinity. In fact, combining \cite[Theorem 2.1 and Corollary 3.1]{SiTa} and \cite[Theorem C]{Rudin} we get:

\begin{prop}\label{Uniq}
For any $u\in L_\frac{1}{2}(\R)$, the Poisson extension $\tilde u$ satisfies $\tilde u(x,y) = o(y^{-1} (x^2+y^2))$ as $|(x,y)|\to \infty$.  Moreover, if $U$ is a harmonic function in $\R^2_+$ which satisfies $U(x,y)=o(y^{-1} (x^2+y^2))$ as $|(x,y)|\to \infty$  and $U(\cdot,y)\to u$ as $y\to 0^+$ in the sense of distributions, then $U = \tilde u$ in $\R^2_+$. 
\end{prop}


\begin{thebibliography}{2}
\bibitem{AT} \textsc{S. Adachi, K. Tanaka}, \emph{Trudinger type inequalities in $\R^N$ and their best exponents}, Proc. Amer. Math. Soc. \textbf{128} (2000), 2051-2057.

\bibitem{Ada} \textsc{D. R. Adams} \emph{A sharp inequality of J. Moser for higher order derivatives}, Ann. of Math. \textbf{128} (1988), no. 2, 385-398.


\bibitem{BGR} \textsc{R. M. Blumenthal, R. Getoor, D. B. Ray,} \emph{On the distribution of first hits for the symmetric stable processes,} Trans. Amer. Math. Soc. \textbf{99} (1961), 540-554.

\bibitem{CaffSil} \textsc{L. Caffarelli, L. Silvestre}, \emph{An extension problem related to the fractional Laplacian}, Comm. Partial Differential Equations \textbf{32} (2007), no. 7-9, 1245–1260.

\bibitem{CC} \textsc{L. Carleson, S.-Y. A. Chang}, \emph{On the existence of an extremal function for an inequality of J. Moser}, Bull. Sci. Math. (2) \textbf{110} (1986) 113-127.

\bibitem{DMR}\textsc{F. Da Lio, L. Martinazzi, T. Rivi\`ere}, \emph{Blow-up analysis of a nonlocal Liouville-type equations}, Analysis and PDE \textbf{8}, no. 7 (2015), 1757-1805.

\bibitem{DlTM} \textsc{A. DelaTorre, G. Mancini}, \emph{Improved Adams-type inequalities and their extremals in dimension $2m$}, preprint (2017), \texttt{arXiv:1711.00892}.

\bibitem{DHMS} \textsc{A. DelaTorre, A. Hyder, L. Martinazzi, Y. Sire}, \emph{The non-local mean-field equation on an interval}, preprint (2018), \texttt{arXiv:1812.02165}.

\bibitem{Dru} \textsc{O. Druet,} \emph{Multibumps analysis in dimension $2$, quantification of blow-up levels}, Duke Math. J. \textbf{132} (2006), 217-269.

\bibitem{FSV} \textsc{A. Fiscella, R. Servadei, E. Valdinoci}, \emph{Density properties for fractional {S}obolev spaces}, Ann. Acad. Sci. Fenn. Math. \textbf{40} (2015), 235--253.

\bibitem{Flu} \textsc{M. Flucher}, \emph{Extremal functions for Trudinger-Moser inequality in $2$ dimensions}, Comment. Math. Helv. \textbf{67} (1992), 471-497.

\bibitem{FM} \textsc{L. Fontana, C. Morpurgo}, \emph{Sharp exponential integrability for critical Riesz potentials and fractional Laplacians on $\R^n$}, Nonlinear Analysis \textbf{167} (2018), 85-122.

\bibitem{Gamma}{\textsc{J. Havil, D. Freeman}, \emph{Gamma: Exploring Euler's Constant}. Princeton University Press (2003), \url{www.jstor.org/stable/j.ctt7sd75}.}


\bibitem{Hyd} \textsc{A. Hyder}, \emph{Structure of conformal metrics on $\mathbb{R}^n$ with constant $Q$-curvature}, preprint (2015), arXiv:1504.07095.

\bibitem{IMNS} \textsc{S. Ibrahim, N. Masmoudi, K. Nakanishi, F. Sani}, \emph{Sharp threshold nonlinearity for maximizing the Trudinger-Moser inequalities}, preprint (2019), arXiv:1902.00958

\bibitem{Ish} \textsc{M. Ishiwata}, \emph{Existence and nonexistence of maximizers
for variational problems associated
with Trudinger-Moser type inequalities in $\R^N$}, Math. Ann. \textbf{351} (2011), 781-804.

\bibitem{IMM} \textsc{S. Iula, A. Maalaoui, L. Martinazzi}, \emph{Critical points of a fractional Moser-Trudinger embedding in dimension $1$},  Differ. Integr. Equ. \textbf{29} (2016), 455-492.

\bibitem{LamLu} \textsc{N. Lam, G. Lu,} \emph{A new approach to sharp Moser-Trudinger and Adams type inequalities: a rearrangement-free argument}, J. Differential Equations \textbf{255} (2013), 298-325.

\bibitem{LL} \textsc{Y. Li, P. Liu}, \emph{A Moser-Trudinger inequality on the boundary of a compact Riemann surface},  Math. Z. \textbf{250} (2005), no. 2, 363-386.

\bibitem{LiRuf} \textsc{Y. Li, B. Ruf}, \emph{A sharp Trudinger-Moser type inequality for unbounded domains in $\R^n$}, Indiana Univ. Math. J. \textbf{57} (2008), 451-480.

\bibitem{LN} \textsc{Y. Li, C. B. Ndiaye}, \emph{Extremal functions for Moser-Trudinger type inequality on compact closed $4$-manifolds}, J. Geom. Anal. \textbf{17} (2007), 669-699.

\bibitem{Lin} \textsc{K. Lin}, \emph{Extremal functions for Moser's inequality}, Trans. Amer. Math. Soc. \textbf{348} (1996), 2663-2671.

\bibitem{LuYang} \textsc{G. Lu, Y. Yang}, \emph{Adams' inequalities for bi-Laplacian and extremal functions in dimension}, Adv. Math. \textbf{220} (2009), 1135-1170.

\bibitem{MMS} \textsc{A. Maalaoui, L. Martinazzi, A. Schikorra}, \emph{Blow-up behavior of a fractional Adams-Moser-Trudinger-type inequality in odd dimension}, Comm. Partial Differential Equations \textbf{41} n. 10, 1593-1618.

\bibitem{MM} \textsc{G. Mancini, L. Martinazzi}, \emph{The Moser-Trudinger inequality and its extremals on a disk via energy estimates}, Calc. Var. Partial Differential Equations \textbf{56} (2017), no. 4, Art. 94, 26.

\bibitem{mar2} \textsc{L. Martinazzi}, \emph{A threshold phenomenon for embeddings of $H^m_0$ into Orlicz spaces}, Calc. Var. Partial Differential Equations \textbf{36} (2009), 493-506.

\bibitem{mar} \textsc{L.Martinazzi}, \emph{Fractional Adams-Moser-Trudinger type inequalities}, Nonlinear Analysis \textbf{127} (2015), 263-278.

\bibitem{mos} \textsc{J. Moser}, \emph{A sharp form of an inequality by N. Trudinger}, Indiana Univ. Math. J. \textbf{20} (1971), 1077-1092.

\bibitem{Ngu} \textsc{V. H. Nguyen}, \emph{A sharp Adams inequality in dimension four and its extremal functions}, Preprint (2017), \texttt{arXiv:1701.08249}.

\bibitem{Park}{\textsc{Park, Y. J}, \emph{Fractional Polya-Szego inquality}, Journal of the Chungcheong Mathematical Society \textbf{24} (2011), No. 2, 267-271}

\bibitem{ON} \textsc{R. O'Neil}, \emph{Convolution operators and $L(p,q)$ spaces}, Duke Math. J. \textbf{30} (1963), 129-142.

\bibitem{Pru} \textsc{A. R. Pruss}, \emph{Nonexistence of maxima for perturbations of some inequalities with critical growth}, Canad. Math. Bull. \textbf{39} (1996), no. 2, 227-237.

\bibitem{Rudin} \textsc{W. Rudin}, \emph{Lectures on the edge-of-the-wedge theorem.},  Conference Board of the Mathematical Sciences Regional Conference Series in Mathematics, No. 6. (1971), Published by the American Mathematical Society.

\bibitem{ruf} \textsc{B. Ruf}, \emph{A sharp Trudinger-Moser type inequality for unbounded domains in $\R^{2}$}, J. Funct. Analysis  \textbf{219} (2004), 340-367.

\bibitem{RS} \textsc{B. Ruf, F. Sani}, \emph{Sharp Adams-type inequalities in $\R^n$}, Trans. Amer. Math. Soc. \textbf{365} (2013), 645-670.

\bibitem{SiTa} \textsc{D. Siegel, E. O. Talvila},  \emph{Uniqueness for the $n$-dimensional half space Dirichlet problem},  Pacific J. Math. \textbf{175} (1996), no. 2, 571--587. 

\bibitem{sIHP} \textsc{M. Struwe}, \emph{Critical points of embeddings of $H^{1,n}_0$
into Orlicz spaces}, Ann. Inst. H. Poincar\'e Anal. Non Lin\'eaire \textbf{5} (1984) 425-464.

\bibitem{Takah} \textsc{F. Takahashi}, \emph{Critical and subcritical fractional Trudinger–Moser-type inequalities on $\mathbb{R}$}, Advances in Nonlinear Analysis (2018), doi:10.1515/anona-2017-0116.

\bibitem{Thi} \textsc{P-D. Thizy}, \emph{When does a perturbed Moser-Trudinger inequality admit an extremal?}, preprint (2018), \texttt{arXiv:1802.01932}.

\end{thebibliography}
\end{document}